\documentclass[letterpaper, 10 pt, conference]{ieeeconf}  

\IEEEoverridecommandlockouts                              

\title{\LARGE \bf
Stability and Performance Analysis of Model Predictive Control of Uncertain Linear Systems
}

\author{Changrui Liu\textsuperscript{$\text{*}$}, Shengling Shi\textsuperscript{$\dagger$} and Bart De Schutter\textsuperscript{$\text{*}$}, \textit{Fellow IEEE}
\thanks{This paper is part of a project that has received funding from the 
European Research Council (ERC) under the European Union’s Horizon 
2020 research and innovation programme (Grant agreement No. 101018826 - CLariNet).}
\thanks{\textsuperscript{$\text{*}$}C. Liu and B. De Schutter are with the Delft Center for Systems and Control, Delft University of Technology, Delft, The Netherlands.
        {\tt\small \{C.Liu-14,B.Deschutter\}@tudelft.nl}}
\thanks{\textsuperscript{$\dagger$}S. Shi is with the Department of Chemical Engineering, Massachusetts Institute of Technology, Cambridge, United States. {\tt\small slshi@mit.edu}}%
}
\usepackage{graphics} 
\usepackage{epsfig} 
\usepackage{svg}
\usepackage{amsfonts}
\usepackage{amsmath}
\usepackage{dsfont}
\usepackage{xcolor} 
\usepackage{bm}
\usepackage{makecell}
\usepackage{hyperref}
\hypersetup{
	colorlinks=true,
	linkcolor={red!50!black},
	citecolor={blue!50!black},
	urlcolor={blue!80!black}
}

\newtheorem{theorem}{Theorem}
\newtheorem{definition}{Definition}
\newtheorem{assumption}{Assumption}
\newtheorem{lemma}{Lemma}
\newtheorem{corollary}{Corollary}
\newtheorem{remark}{Remark}
\newtheorem{proposition}{Proposition}

\renewcommand{\QED}{\QEDopen}

\newcommand{\cliu}[1]{{\texttransparent{1}{\color{black}#1}}}
\newcommand{\qt}[1]{{\texttransparent{1}{\color{black}#1}}}

\newcommand{\bR}{\mathbb{R}}
\newcommand{\bZ}{\mathbb{Z}}
\newcommand{\bN}{\mathbb{N}}

\newcommand{\cX}{\mathcal{X}}
\newcommand{\cU}{\mathcal{U}}

\newcommand{\cA}{\mathcal{A}}
\newcommand{\cB}{\mathcal{B}}

\newcommand{\splus}{\hspace{-0.08cm}+\hspace{-0.08cm}}
\newcommand{\sminus}{\hspace{-0.08cm}-\hspace{-0.08cm}}
\newcommand{\seq}{\hspace{-0.08cm}=\hspace{-0.08cm}}
\newcommand{\sdeq}{\hspace{-0.08cm}:=\hspace{-0.08cm}}
\newcommand{\sless}{\hspace{-0.08cm}<\hspace{-0.08cm}}
\newcommand{\spd}{.\hspace{-0.15cm}}
\newcommand{\scm}{,\hspace{-0.15cm}}

\newcommand{\upQ}{\overline{\sigma}_Q}

\newcommand{\xb}{x_{\mathrm{bd}}}

\newcommand{\ub}{u_{\mathrm{bd}}}

\newcommand{\hA}{\widehat{A}}
\newcommand{\hB}{\widehat{B}}
\newcommand{\hK}{\widehat{K}}

\newcommand{\eA}{\delta_{A}}
\newcommand{\eB}{\delta_{B}}

\newcommand{\bu}{\mathbf{u}}

\newcommand{\otx}{\psi_x}
\newcommand{\oex}{\widehat{\psi}_x}

\newcommand{\ctx}{\phi_x}
\newcommand{\cex}{\widehat{\phi}_x}

\newcommand{\ctxk}{\phi^{(\kappa)}_x}

\newcommand{\uit}{\mathbf{u}^\ast_{\infty}}
\newcommand{\Vit}{V_\infty}

\newcommand{\un}{\bu^\ast_{N}}
\newcommand{\Vn}{V_{N}}

\newcommand{\hun}{\widehat{\bu}^\ast_{N}}
\newcommand{\hVn}{\widehat{V}_{N}}
\newcommand{\bvn}{\mathbf{v}_{N}}

\newcommand{\hmuN}{\widehat{\mu}_N}
\newcommand{\muN}{\mu_N}

\newcommand{\Vf}{V_{\mathrm{f}}}

\newcommand{\nf}{N_{\mathrm{f}}}

\newcommand{\fx}{\mathcal{X}_{\mathrm{tnd}}}

\newcommand{\xp}{x^{+}_{\mathrm{re}}}

\newcommand{\xroa}{\mathcal{X}_{\mathrm{ROA}}}

\newcommand{\du}{\delta\mathbf{u}}

\newcommand{\nQ}[1]{\| #1 \|^2_Q}
\newcommand{\nR}[1]{\| #1 \|^2_R}
\newcommand{\ntwo}[1]{\| #1 \|_2}
\newcommand{\Acl}{A_{\mathrm{cl}}}
\newcommand{\hAcl}{\widehat{A}_{\mathrm{cl}}}

\newcommand{\Epsi}{E_{N,(\psi)}}
\newcommand{\Epsiu}{E_{N,(\psi,u)}}
\newcommand{\Eu}{E_{N,(u)}}

\newcommand{\hPhiN}{\widehat{\Phi}_N}
\newcommand{\PhiN}{\Phi_N}
\newcommand{\hGaN}{\widehat{\Gamma}_N}
\newcommand{\GaN}{\Gamma_N}
\newcommand{\bQN}{\overline{Q}_N}
\newcommand{\bQNp}{\overline{Q}_{N+1}}
\newcommand{\bRN}{\overline{R}_N}
\newcommand{\hHN}{\widehat{H}_N}
\newcommand{\hXi}{\widehat{\Xi}}
\newcommand{\hzeta}{\widehat{\zeta}}

\newcommand{\myratio}{\frac{ \delta^2_B \underline{\sigma}_Q }{ \delta^2_A \underline{\sigma}_R }}

\newcommand{\invmyratio}{\frac{ \delta^2_A \underline{\sigma}_R }{ \delta^2_B \underline{\sigma}_Q }}

\begin{document}

\maketitle
\thispagestyle{empty}
\pagestyle{empty}

\begin{abstract}
Model mismatch often poses challenges in model-based controller design. This paper investigates model predictive control (MPC) of uncertain linear systems with input constraints, focusing on stability and closed-loop infinite-horizon performance. The uncertainty arises from a parametric mismatch between the true and the estimated system under the matrix Frobenius norm. We examine a simple MPC controller that exclusively uses the estimated system model and establishes sufficient conditions under which the MPC controller can stabilize the true system. Moreover, we derive a theoretical performance bound based on relaxed dynamic programming, elucidating the impact of prediction horizon and modeling errors on the suboptimality gap between the MPC controller and the oracle infinite-horizon optimal controller with knowledge of the true system. Simulations of a numerical example validate the theoretical results. Our theoretical analysis offers guidelines for obtaining the desired modeling accuracy and choosing a proper prediction horizon to develop certainty-equivalent MPC controllers for uncertain linear systems.
\end{abstract}

\section{Introduction}
\label{sec:introduction}
Model predictive control (MPC) is an optimization-based control strategy that computes inputs to optimize a specific performance metric over a given prediction horizon based on a system model. MPC has found widespread application in various fields, such as chemical processes \cite{santander2016economic}, aerospace vehicles \cite{eren2017model}, and portfolio optimization \cite{dombrovskii2015model}. Regardless of the application, 
the effectiveness of MPC heavily depends on the accuracy of the prediction model. However, obtaining a perfect model is impossible due to inherent modeling errors in practice.

To address the issues that arise from having an imperfect model, adaptive MPC is commonly employed, integrating MPC with a system identification module \cite{adetola2009adaptive}. Typical approaches use the comparison error \cite{fukushima2007adaptive}, neural networks \cite{son2022learning}, and Bayesian inference \cite{dogan2023regret}. On the other hand, data-driven MPC \cite{berberich2020data} has also emerged as a promising method for handling model uncertainty by directly using input-output data. However, while much of the literature focuses on the stability, feasibility, and robustness of MPC, studies that provide performance analysis are limited.

Relaxed dynamic programming (RDP) is a notable framework for analyzing the performance of MPC controllers compared to that of the idealized infinite-horizon optimal control problem \cite{lincoln2006relaxing}. RDP quantifies the suboptimality gap by analyzing a general value function that describes the energy-decreasing characteristic along the closed-loop system trajectory \cite{giselsson2010adaptive, grune2008infinite}. For uncertain systems, recent extensions analyze the effect of disturbances for nonlinear systems \cite{schwenkel2020robust} and a class of parameterized linear systems where the true matrices of the system lie in a known polytope \cite{moreno2022performance}. However, performance analysis of MPC for general linear systems with modeling errors remains unexplored.

\textit{Contributions}: This work presents a novel analysis of MPC performance for linear systems with modeling errors. In contrast to the previous work in \cite{moreno2022performance}, we do not assume any specific parametric structure of the system nor adapt the system model online, adhering to the framework of model-based control with offline system identification. Moreover, our established bound is a \textit{consistent} extension of the bounds derived in \cite{grune2008infinite, kohler2021robust}, allowing it to recover the case without model mismatch. Using the RDP method, we establish a theoretical performance bound illustrating the impact of modeling errors on the closed-loop performance of the MPC controller. Furthermore, we provide sufficient conditions on the prediction horizon in the presence of modeling errors such that the closed-loop system is stable. This further reveals how closed-loop performance depends on modeling errors as well as the prediction horizon.

The remainder of the paper is organized as follows. After providing the preliminaries in Section \ref{sec:preliminaries}, we formulate the optimal control problems in Section \ref{sec:problem_formulation}. Section \ref{sec:theoretical_analysis} offers stability and performance analysis of the MPC controller, followed by a numerical example in Section \ref{sec:numerical example} to validate the theoretical results.
\section{Preliminaries}
\label{sec:preliminaries}

\subsection{Notation}
Let $x$ be a vector; then its transpose is denoted by $x^\top$, and its vector $i$-norm by $\|x\|_i$. For a matrix $M$, $M^\top$, $\|M\|_2$, $\rho(M)$ and $\|M\|_{\mathrm{F}}$ denote its transpose, matrix $2$-norm (i.e., spectral norm), spectral radius and Frobenius norm, respectively. Moreover, $M \geq 0 \;(x \geq 0)$ indicates element-wise nonnegativity, while $M \succ (\succeq) \;0 $ indicates positive (semi)definiteness. For a \textit{symmetric} matrix $M \succ 0$, its largest and smallest eigenvalues are, respectively, denoted by $\overline{\sigma}_M$ and $\underline{\sigma}_M$, and we further define $r_M := \frac{\overline{\sigma}_M}{\underline{\sigma}_M}$. For a vector $x$ and a positive symmetric (semi)definite matrix $M$, $\|x\|_M$ stands for $(x^\top M x)^{1/2}$.

The set of natural numbers is denoted by $\bN$, \cliu{the set of positive integers up to $n$ 
is denoted by $\mathbb{Z}_n^{+}$ with $\mathbb{Z}_{\infty}^{+}$ representing the set of all positive integers}, and the set of real and non-negative real numbers are denoted, respectively, by $\bR$ and $\bR_+$. We will sometimes use the \textit{bold} letter $\mathbf{x}$ to represent concatenation of a sequence of vectors $\{x_i\}$ as $\mathbf{x} = [x^\top_0, x^\top_1, \dots]^\top$, and $\mathbf{x}[i] := x_i$. For any two vectors (matrices) $x$ and $y$, $x \otimes y$ stands for their Kronecker product. Moreover, $\mathbf{0}_n$, $\mathds{1}_n$, and $I_n$ are the zero vector, one vector, and identity matrix of dimension $n$, respectively. \cliu{Finally, $\lceil\cdot\rceil$ is the standard ceiling operator (i.e., the least integer operator).}
\subsection{System Description \& Definitions}
We consider discrete-time linear time-invariant (LTI) systems given by
\begin{equation}
    \label{eq:sys_linear_true}
    x(t+1) = A x(t) + B u(t), \quad t \in \bN,
\end{equation}
where \cliu{$x(t) \in \cX = \bR^n$} is the state, $u(t) \in \cU \subseteq \bR^m$ is the input with $\cU$ being the input constraint set, and $A \in \bR^{n \times n}$ and $B \in \bR^{n \times m}$ are the matrices of the \textit{true} system. \cliu{We use the notation $x(t)$ ($u(t)$) to denote the true state (input) at time step $t$, whereas $x_t$ ($u_t$) represents the predicted state (input) and/or decision variables in optimization problems.} In this paper, we consider a stabilizable pair $(A,B)$, which is a standard assumption in the field \cite{moreno2022performance, boccia2014stability}. The set $\cU$ is described using a set of linear constraints as
\begin{equation}
    \label{eq:cons_u}
    \cU = \{u \in  \bR^m \mid  F_u u \leq \mathds{1}_{c_u} \},
\end{equation}
where $F_u \in \bR^{c_u \times m}$ with $c_u$ the number of input constraints. \cliu{Note that $\mathbf{0}_m \in \cU$}, and hence $\cU$ is nonempty. 
\cliu{For the model given in as \eqref{eq:sys_linear_true}}, the \textit{open-loop} predicted state, starting from any state $x \in \cX$ and being predicted $k$ steps forward under the control sequence \cliu{$\mathbf{u} = [u^\top_0, u^\top_1, \dots, u^\top_{k-1}]^\top$}, is denoted as $\otx(k,x,\bu)$. For linear systems characterized by $(A, B)$, it is commonly known that
\begin{equation}
    \label{eq:state_evolution}
    \otx(k,x,\bu) = A^{k}x + \sum^{k-1}_{i=0}A^{k-1-i}B\bu[i].
\end{equation}
On the other hand, given a state-feedback control law $\mu: \cX \to \cU$, we denote the corresponding \textit{closed-loop} predicted state, starting from any state $x \in \cX$ and being predicted $k$ steps forward, as $\ctx^{[\mu]}(k,x)$.
For the system described in \eqref{eq:sys_linear_true}, the controller only has access to an estimated system governed by the matrices $\hat{A} \in \cA(A, \eA) \subseteq \bR^{n \times n}$ and $\hat{B} \in \cB(B, \eB) \subseteq \bR^{n \times m}$, where the uncertainty sets $\cA(A, \eA)$ and $\cB(B, \eB)$ are defined as
\begin{subequations}
 \label{eq:uncertainty_set}
\begin{align}
    \label{eq:uncertainty_set_A}
    \cA(A, \eA) &= \{M \in  \bR^{n \times n} \mid  \|A - M\|_{\mathrm{F}} \leq \eA\},\\
    \label{eq:uncertainty_set_B}
    \cB(B, \eB) &= \{M \in  \bR^{n \times m} \mid \|B - M\|_{\mathrm{F}} \leq \eB\},
\end{align}
\end{subequations}
where the parameters $\eA \geq 0$ and $\eB \geq 0$ are characterized using system identification or machine learning techniques before initiating the control task. For the open-loop and closed-loop predicted state of the estimated system, we use the notation $\oex(k,x,\bu)$ and $\cex^{[\mu]}(k,x)$, respectively. \cliu{In addition, we impose the following standard assumption for linear systems:
\begin{assumption}
    The pairs $(A, B)$ and $(\hA, \hB)$ are both stabilizable.
\end{assumption}}
Finally, for convenience in the analysis in Section \ref{sec:theoretical_analysis}, we provide the following definition:
\begin{definition}[Error-consistent function]
\label{def:error_consistent}
    We call a function $\alpha(\cdot, \cdot): \bR^2_{+} \to \bR_{+}$ to be error-consistent if the following conditions hold:
    \begin{enumerate}
        \item $\alpha(\delta_1, \cdot)$ is non-decreasing for any $\delta_1 \in \bR^+$,
        \item $\alpha(\cdot, \delta_2)$ is non-decreasing for any $\delta_2 \in \bR^+$,
        \item $\alpha(\delta_1, \delta_2) = 0$ if and only if $\delta_1 = \delta_2 = 0$.
    \end{enumerate}
\end{definition}

\section{Problem Formulation}
\label{sec:problem_formulation}
We consider an infinite-horizon optimal control problem, in which the controller aims to generate an input sequence $\bu$ that stabilizes the system (i.e., steers the state to the origin) while minimizing the performance metric
\begin{equation}
\label{eq:performance_infinite}
    J_{\infty}(x, \bu_{\infty}) := \sum^\infty_{t=0} l(x_t, u_t),
\end{equation}
where $x_0 = x \in \cX$ and \cliu{$\mathbf{u}_{\infty} = [u^\top_0, u^\top_1, \dots, u^\top_{\infty}]^\top$}. In this work, we consider a quadratic stage cost given by $l(x, u) = \nQ{x} + \nR{u}$, where $Q \in \bR^{n \times n}$ and $R \in \bR^{m \times m}$ are symmetric matrices satisfying $Q, R \succ 0$, and we define $l^\ast(x) = \min_{u \in \cU}l(x, u) = \nQ{x}$. In addition, for any linear control law $u = \kappa(x) = Kx$, we define the local region $\Omega_{K} = \{x \in \cX \mid l^\ast(x) \leq \varepsilon_{K} \}$ with $\varepsilon_{K} > 0$ such that $Kx \in \cU$ for all $x \in \Omega_{K}$. \cliu{The maximum $\epsilon_K$ can be derived analytically, and} the explicit form of which is given as 
\begin{equation}
\label{eq:computation_ellipsoidal_radius}
    \varepsilon_K = \min_{i}\frac{1}{\|[F_uK]_{(i,:)}^\top\|^2_{Q^{-1}}},
\end{equation}
where $[F_uK]_{(i,:)}$ denotes the $i$-th row of the matrix $F_uK$.
Given the dynamics in \eqref{eq:sys_linear_true} and the input constraint set in \eqref{eq:cons_u}, we consider the following optimization problem:
        \begin{align*}
        \mathrm{P}_{\mathrm{IH-OCP}}: & \min_{\{u_{t}\}^{\infty}_{t = 0}} \sum^{\infty}_{t=0}\nQ{x_t} + \nR{u_t} \\
        \text{s.t. } & \; x_{t+1} = Ax_{t} + Bu_{t}, \forall t \in \bN, \\
        & \; u_{t} \in \mathcal{U}, \forall t \in \bN,\\
        & \; x_{0} = x(0). 
        \end{align*}
\cliu{We denote the optimal solution to the problem $\mathrm{P}_{\mathrm{IH-OCP}}$ by $\uit(x(0))$; the associated optimal value of the cost function is $\Vit(x(0)) := J_{\infty}\left(x(0), \uit(x(0))\right)$}.
\begin{assumption}
\label{ass:initial_controllability}
    The initial state $x(0)$ lies in the region of attraction $\xroa$ of the considered system in \eqref{eq:sys_linear_true} such that, given the input constraints in \eqref{eq:cons_u}, $\Vit(x(0))$ is \textit{finite} for all $x(0) \in \xroa$.
\end{assumption}
\begin{corollary}
\label{corollary:control_invariant_region_of_contraction}
    Under Assumption \ref{ass:initial_controllability}, $\xroa$ is a control invariant set, i.e., for all $x \in \xroa$, there exists $u \in \cU$ such that $Ax + Bu \in \xroa$.
\end{corollary}
However, due to the infinite nature of $\mathrm{P}_{\mathrm{IH-OCP}}$, computing the optimal input is intractable. \cliu{Therefore, at each time step $t$, we consider an truncated performance metric as}
\begin{equation}
\label{eq:performance_mpc}
    J_{N}(x(t), \bu_N) = \sum^{N}_{k=0} \nQ{x_{k|t}} + \nR{u_{k|t}},
\end{equation}
where $N \geq 1$ is the prediction horizon, $x_{k|t}$ and $u_{k|t}$ are, respectively, the $k$-step forward predicted state and input with $x_{0|t} = x(t)$ being the true state at time step $t$, \cliu{and $\bu_N = \{u^\top_{0|t}, u^\top_{1|t}, \dots, u^\top_{N-1|t}\}^\top$}. In this work, we consider the formulation without a terminal cost, as also done in \cite{kohler2021robust, muller2017quadratic}.
At each time step $t$, the \textit{ideal} MPC controller that has access to the true system solves the following optimization problem:
        \begin{align*}
        \mathrm{P}_{\mathrm{ID-MPC}}: & \min_{\{u_{k|t}\}^{N}_{k = 0}} \sum^{N}_{k=0} \left(\nQ{x_{k|t}} + \nR{u_{k|t}}\right) 
        \end{align*}
        \vspace{-0.35cm}
        \begin{align*}
        \text{s.t.} \; &\; x_{k+1|t} = Ax_{k|t} + Bu_{k|t}, \forall k \in \bZ^{+}_{N-1},  \\
        & \; u_{k|t} \in \mathcal{U}, \forall k \in \bZ^{+}_{N}, \\
         & \; x_{0|t} = x(t). 
        \end{align*}

\cliu{By solving the problem $\mathrm{P}_{\mathrm{ID-MPC}}$, the optimal value of $\bu_N$ is denoted by $\un(x(t))$}, and the corresponding value of the cost function is $\Vn(x(t)) := J_{N}(x(t), \un(x(t)))$. Moreover, the problem $\mathrm{P}_{\mathrm{ID-MPC}}$ implicitly defines an ideal MPC control law as $\mu_N(x(t)) := \un(x(t))[0]$. \cliu{In practice, however, a nominal} MPC controller can only rely on the \textit{estimated} system, and it instead solves the following optimization problem:
        \begin{align*}
        \mathrm{P}_{\mathrm{NM-MPC}}: & \min_{\{u_{k|t}\}^{N}_{k = 0}} \sum^{N}_{k=0} \left(\nQ{x_{k|t}} + \nR{u_{k|t}}\right) 
        \end{align*}
        \vspace{-0.35cm}
        \begin{align*}
        \text{s.t.} \; &\; x_{k+1|t} = \hA x_{k|t} + \hB u_{k|t}, \forall k \in \bZ^{+}_{N},  \\
        & \; u_{k|t} \in \mathcal{U}, \forall k \in \bZ^{+}_{N}, \\
         & \; x_{0|t} = x(t). 
        \end{align*}
Likewise, we denote the optimal solution to the problem $\mathrm{P}_{\mathrm{NM-MPC}}$ by $\hun(x(t))$, the value of the cost function follows as $\hVn(x(t)) := J_{N}(x(t), \hun(x(t)))$, and the real MPC control law is defined as $\hat{\mu}_N(x(t)) := \hun(x(t))[0]$.
\cliu{Starting from $x(0) = x$, we apply the nominal MPC controller recursively, and the resulting performance is}
\begin{equation}
    \label{eq:MPC_infinite_time_performace}
    J^{[\hmuN]}_{\infty}(x) = \sum^\infty_{t = 0}l\left(\ctx^{[\hmuN]}(t,x), \hmuN\big(\ctx^{[\hmuN]}(t,x)\big)\right)
\end{equation}
The main objectives of this paper are twofold: (i) to investigate under which conditions the MPC control law $\hmuN$ can stabilize the system and (ii) to quantify the closed-loop performance of the real MPC controller $J^{[\hmuN]}_{\infty}(x)$ relative to the \cliu{optimal infinite-horizon cost $\Vit (x)$ (i.e., the optimal value function $\Vit$ evaluated at $x(0) = x$)}.
\begin{remark}[Nullified input at stage $N$]
\label{remark:zero_final_input}
    For both the optimization problems $\mathrm{P}_{\mathrm{ID-MPC}}$ and $\mathrm{P}_{\mathrm{RE-MPC}}$, the optimal input satisfies $u^\ast_{N|t} = 0$ due to the positive definiteness of the quadratic cost. We include $u_{N|t}$ in our formulation to be consistent with the one without a terminal cost.
\end{remark}
\begin{remark}[State constraints \& region of attraction]
    In this work, we do not consider an explicit state constraint set $\cX$. However, in the presence of hard input constraints, stabilizing the system may not be possible for any arbitrary initial state $x(0) = x$, especially for unstable systems with $\rho(A) \geq 1$. Therefore, we impose Assumption \ref{ass:initial_controllability} for the validity of our work. Similar assumptions have been made in \cite{kohler2021stability} and \cite{kohler2023stability} using the notion of cost controllability. Characterizing $\xroa$ without knowing the true system is an open issue and is out of the scope of this paper.
\end{remark}

\section{Theoretical Analysis}
\label{sec:theoretical_analysis}
In this section, we provide an analysis of the stability and the closed-loop performance of the MPC controller with model mismatch. In Section \ref{subsec:mpc_value_function}, we establish a relation between the MPC value function $\hVn$ and the infinite-horizon optimal value function $\Vit$, leveraging the sensitivity analysis of quadratic programs (QPs). In Section \ref{subsec:stability_analysis}, we theoretically derive a performance bound using the RDP inequality. 

\subsection{Evaluation of the MPC value function}
\label{subsec:mpc_value_function}
We first establish an upper bound of $\hVn$ in terms of $\Vit$. Note that these two value functions are constructed using different dynamic models. We first state the main result and highlight its interpretations.
\begin{proposition}
\label{prop:bounding_mpc_cost}
    There exist two error-consistent functions $\alpha_N(\eA, \eB)$ and $\beta_N(\eA, \eB)$ such that, for all $x \in \xroa$, $\hVn$ and $\Vit$ satisfy the following inequality\footnote{The explicit dependence of $\alpha_N$ and $\beta_N$ on the prediction horizon is highlighted using the subscript $N$.}:
\begin{equation} \label{eq:bounding_relation_general} \hVn (x) \leq \big (1 + \alpha_N(\eA, \eB)\big) \Vit(x) + \beta_N(\eA, \eB),
\end{equation}
where functions $\alpha_N$ and $\beta_N$ relate to the eigenvalues and matrix norms of $\hA$, $\hB$, $Q$, $R$, the input constraint set $\cU$, but not to quantities derived from $A$ or $B$. Explicit expressions of $\alpha_N$ and $\beta_N$ are given in Appendix \ref{appendix:C--bounding_mpc}.
\end{proposition}

The upper bound in \eqref{eq:bounding_relation_general} is consistent with the bound without model mismatch: $\Vn(x) \leq \Vit(x)$. In fact, if $\eA = \eB = 0$, $\hVn$ is identical to $\Vn$ and \eqref{eq:bounding_relation_general} degenerates to $\Vn(x) \leq \Vit(x)$ since both $\alpha_N$ and $\beta_N$ are error-consistent. Furthermore, $\alpha_N$ and $\beta_N$ are \textit{computable} because they do not depend on the matrix pair $(A, B)$ of the true system. 
We provide a proof sketch of Prop.~\ref{prop:bounding_mpc_cost} to highlight the main steps, and more details can be found in Appendix \ref{appendix:C--bounding_mpc}.

\noindent \textbf{Step 1):} Expanding $\hVn$ in terms of $\oex(\cdot,x,\hun(x))$ and $\hun(x)$. By definition, we have
\begin{equation*}
    \hVn(x) = \sum^{N}_{k = 0} \left(\nQ{\oex(k, x, \hun(x))} + \nR{\big(\hun(x)\big)[k]} \right).
\end{equation*}

\noindent \textbf{Step 2):} Keeping the input $\hun(x)$ unchanged and decomposing $\oex(\cdot,x,\hun(x))$ into $\otx(\cdot,x,\hun(x))$ (i.e., the open-loop predicted state using the true system model under input $\hun(x)$) and $e_{\psi}(\cdot) := \otx(\cdot,x,\hun(x)) - \oex(\cdot,x,\hun(x))$ (i.e., the open-loop prediction error under input $\hun(x)$). The remaining task is to derive an upper bound for $\nQ{e_{\psi}(\cdot)}$, which further requires bounding the terms $\ntwo{A^k - \hA^k}$ and $\ntwo{A^kB - \hA^k\hB}$ for $k \in \bZ^+_{N}$.

\noindent \textbf{Step 3):} Decomposing the input $\hun(x)$ into $\un(x)$ (i.e., the optimal solution of the problem $\mathrm{P}_{\mathrm{ID-MPC}}$) and $\du(x) := \un(x) - \hun(x)$. Similar to Step 2, the main difficulty is to derive an upper bound for $\ntwo{\du(x)}$. We first transform the problems $\mathrm{P}_{\mathrm{ID-MPC}}$ and $\mathrm{P}_{\mathrm{RE-MPC}}$, respectively, into their QP formulations, 
then apply the results in the sensitivity analysis of QPs to provide an upper bound as a function of $\delta_A$ and $\delta_B$.

\noindent \textbf{Step 4):} Decomposing $\otx(\cdot,x,\hun(x))$ into $\otx(\cdot,x,\un(x))$ (i.e., the open-loop predicted state using the true system model under input $\un(x)$) and $\otx(\cdot,0,\du(x))$ (i.e., the open-loop predicted state, starting from $x = 0$, using the true system model under input $\du(x)$). The core task is to obtain an upper bound for $\nQ{\otx(\cdot,0,\du(x))}$, which further reduces to obtaining an upper bound for $\nQ{\du(x)}$ based on linearity. Therefore, the intermediate results of Step 3 can be used again.

\cliu{Following the above steps, we can express $\hVn(x)$ in terms of $\un(x)$ and $\otx(\cdot,x,\un(x))$, which leads to a relationship between $\hVn(x)$ and $\Vn(x)$. Consequently, by using the fact that $\Vn(x) \leq \Vit(x)$, we can arrive at the desired inequality as in \eqref{eq:bounding_relation_general}.}
\subsection{Closed-loop Stability and Performance Analysis}
\label{subsec:stability_analysis}
Under the MPC control law $\hmuN$, at a given state $x$, the next-step state is computed as $\xp = Ax + B\hmuN(x)$. If there exists a so-called energy function $\tilde{V}: \bR^{n} \to \bR_+$ such that the RDP inequality
\begin{equation}
    \label{eq:rdp_prinstine}
    \tilde{V}(\xp) - \tilde{V}(x) \leq -\epsilon l(x, \hmuN(x))
\end{equation}
holds for all $x \in \xroa$, where $\epsilon \in (0,1]$, then the controlled system is asymptotically stable \cite{lincoln2006relaxing, grune2017nonlinear}. In this paper, we investigate the case where $\tilde{V} = \hVn$ and theoretically derive the coefficient $\epsilon$ as a function of the prediction horizon $N$ and the parameters $\eA$ and $\eB$ that quantify the model mismatch.

We first present a preliminary result that can be applied to both the true system and the estimated system. 
\begin{lemma}
\label{lm:local_stabilization_property}
    \cliu{Given a positive-definite quadratic stage cost $l(x,u) = \nQ{x} + \nR{u}$ and a local linear stabilizing control law $u = \kappa(x) = Kx$ for the system pair $(\hA, \hB)$}, there exist scalars \cliu{$\lambda_{K} \geq 1$} and $\rho_K \in (0,1)$ such that for all $x \in \Omega_{K}$ and $k \in \mathbb{N}$ we have
    \begin{equation}
        \label{eq:local_exponential}
        l(\cex^{[\kappa]}(k,x), K\cex^{[\kappa]}(k,x)) \leq C^\ast_{K} (\rho_K)^kl^\ast(x),
    \end{equation}
where \cliu{$C^\ast_{K} = \big(1+\underline{\sigma}^{-1}_Q\overline{\sigma}_R\|K\|^2_2\big)r_Q(\lambda_K)^2$}.
\end{lemma}
\begin{proof}
    \cliu{The proof is closely related to the common results on exponential stability of linear systems, and it is given in Appendix \ref{appendix:B--Lemma_exponential decay} for completeness.}
\end{proof}

Based on Lemma \ref{lm:local_stabilization_property}, we have the following result about the properties of the open-loop system:
\begin{lemma}
\label{lm:terminal_state_exponential_bound}
    Given an upper bound $M_{\hat{V}}$ of $\hVn$ for all $x = x(0)$ and a stabilizing control law $u = Kx$ for the estimated system, there exist constants $L_{\widehat{V}} := \max\{\gamma, \frac{M_{\widehat{V}}}{\varepsilon_K}\}$ and $N_0 := \big\lceil \max\{0, \frac{M_{\widehat{V}}-\gamma \varepsilon_K}{\varepsilon_K}\}\big\rceil$ such that for $N \geq N_0$ we have
    \begin{subequations}
        \begin{align}
        \label{eq:ratio_bound_hatvN}
            & \hVn(x) \leq L_{\widehat{V}}l^\ast(x) \leq L_{\widehat{V}} l(x, \hmuN(x)) \\
        \label{eq:exponential_bound_final_stage}
            & \nQ{\oex(N, x, \hun(x))} \leq \gamma \rho^{N - N_0}_\gamma l(x, \hmuN(x)),
        \end{align}
    \end{subequations}
    \cliu{where $\varepsilon_{K}$ is defined as in \eqref{eq:computation_ellipsoidal_radius} using the gain $K$, $\gamma = C^\ast_K(1 - \rho_K)^{-1}$ with $C^\ast_{K}$ and $\rho_K$ given as in Lemma \ref{lm:local_stabilization_property}, and $\rho_{\gamma} = \gamma^{-1}(\gamma - 1)$.}
\end{lemma}
\begin{proof}
    The proof follows a procedure similar to that in \cite[Theorem~5]{kohler2021stability} and is based on the lemma \ref{lm:local_stabilization_property}.
\end{proof}
For a stabilizable system, \eqref{eq:ratio_bound_hatvN} and \eqref{eq:exponential_bound_final_stage} provides an upper bound, respectively, for the value function $\hVn$ and the cost of the final state, both in terms of the stage cost. Moreover, \eqref{eq:exponential_bound_final_stage} implies that, given a sufficiently long horizon, the cost of the final state exponentially decays as the horizon increases. The \textit{critical horizon} $N_0$ quantifies the required number of steps such that the final state lies in the local region $\Omega_K$.

\begin{proposition}
\label{prop:energy_decreasing}
    There exist a constant $\eta_N$ and an error-consistent function $\xi_N$ satisfying $\xi_N(\delta_A, \delta_B) + \eta_N < 1$, and for all $x \in \xroa$ we have
    \begin{multline}
    \label{eq:energy_decreasing}
        \hVn(\xp) - \hVn(x) \leq \\
        -(1 - \xi_N(\delta_A, \delta_B) - \eta_N) l(x, \hat{\mu}_N(x)),
    \end{multline}
    where the function $\xi_N$ and constant $\eta_N$ relate to the eigenvalues and matrix norms of $\hA$, $\hB$, $Q$, $R$, the input constraint set $\cU$, but not to quantities derived from $A$ or $B$. The explicit expressions of $\xi_N$ and $\eta_N$ are given, respectively, in \eqref{eq:xi_N} and \eqref{eq:eta_N} in Appendix \ref{appendix:D--energy_decreasing}.
\end{proposition}

In Proposition \ref{prop:energy_decreasing}, the relation in \eqref{eq:energy_decreasing}, serving as the RDP inequality for the closed-loop systems (cf. \eqref{eq:rdp_prinstine}), provides a lower bound of the energy decrease in terms of the stage cost. Moreover, based on $\xi_N(\delta_A, \delta_B) + \eta_N < 1$, a \textit{sufficient} condition on the prediction horizon and modeling error can be derived such that the closed-loop system is stable. We summarize this condition in the following corollary:
\begin{corollary}
\label{corollary:small_mismatch}
Given a sufficiently long prediction horizon $N$ satisfying
\cliu{
\begin{equation}
\label{eq:horizon_requirements_original}
    N > N_0 + \frac{\log\left(\|\hA\|^2_2r_Q\gamma\right)}{\log(\rho_{\gamma}^{-1})},
\end{equation}
where $N_0$, $\gamma$ and $\rho_\gamma$ are defined as in Lemma \ref{lm:terminal_state_exponential_bound} using gain $K$, the closed-loop system is asymptotically stable if the modeling error is small enough such that}
\begin{equation}
    \hspace*{-2ex} h(\eA,\eB) \sless \left\{\frac{-\omega_{N,(\frac{1}{2})} \splus [\omega^2_{N,(\frac{1}{2})} \splus \omega_{N,(1)}(1\sminus \eta_N)]^{\frac{1}{2}}}{\omega_{N,(1)}} \right\}^2 \hspace*{-1ex}\scm 
\end{equation}
where $h$ is defined as in \eqref{eq:prediction_error_bound_h_function} in Lemma \ref{lm:prediction_error_bound}, $\eta_N$ is given as in \eqref{eq:eta_N}, and $\omega_{N,(1)}$ and $\omega_{N,(\frac{1}{2})}$ are given, respectively, as in \eqref{eq:omega_1} and \eqref{eq:omega_0.5}.
\end{corollary}

Informally, Corollary \ref{corollary:small_mismatch} indicates that the model mismatch should be small enough to ensure the stability of the closed-loop system if the MPC control law is derived from the estimated model instead of the true model. Next, we provide a proof sketch of Proposition \ref{prop:energy_decreasing}. The main part of the proof is to establish the upper bound for $\hVn(\xp) - \hVn(x)$ in terms of $l(x, \hmuN(x))$, and it consists of four steps:

\noindent \textbf{Step 1):} Expanding $\hVn(\xp)$ in terms of $\hun(\xp)$ and $\oex(\cdot,\xp,\hun(\xp))$, and expanding $\hVn(x)$ in terms of $\hun(x)$ and $\oex(\cdot,x,\hun(x))$.

\noindent \textbf{Step 2):} Constructing an auxiliary input $\mathbf{v}_N(x)$ satisfying $\big(\mathbf{v}_N(x)\big)[i-1] = \big(\hun(x)\big)[i]$ for $i \in \bZ^+_{N-1}$, $\big(\mathbf{v}_N(x)\big)[N-1] = 0$, and $\big(\mathbf{v}_N(x)\big)[N] = 0$, and then relaxing $\hVn(\xp)$ as $J_N(\xp, \mathbf{v}_N(x))$. This auxiliary input has two benefits when computing the difference $J_N(\xp, \mathbf{v}_N(x)) - \hVn(x)$:
\begin{enumerate}
    \item The input-incurred cost $\nR{\big(\mathbf{v}_N(x)\big)[i-1]}$ and $\nR{\big(\hun(x)\big)[i]}$ cancel each other;
    \item We can prove that $\oex(i-1,\xp,\mathbf{v}_N(x)) = \hA^{i-1} \Delta x + \oex(i,x,\hun(x))$ for all $i \in \bZ^+_{N}$, where $\Delta x = (A - \hA)x + (B - \hB)\hmuN(x)$ is the one-step-ahead prediction error under the MPC control law $\hmuN$.
\end{enumerate}

\noindent \textbf{Step 3):} Reorganizing the terms obtained from Step 2 and applying Lemma \ref{lm:quadratic_norm}, we build a relation between $\nQ{\hA^{i-1} \Delta x}$ and $l(x, \hmuN(x))$. 

\noindent \textbf{Step 4):} Using the property given in Remark \ref{remark:zero_final_input}, the remaining term is $\nQ{\oex(N-1,\xp,\mathbf{v}_N(x))}+\nR{\big(\mathbf{v}_N(x)\big)[N-1]} + \nQ{\oex(N,\xp,\mathbf{v}_N(x))} - \nQ{\oex(N, x, \hun(x))}$, which can be further bounded using Lemma \ref{lm:local_stabilization_property} and Lemma \ref{lm:terminal_state_exponential_bound}.

The final infinite-horizon performance guarantee is stated in the following theorem:
\begin{theorem}
Given $\alpha_N$ and $\beta_N$ as in Proposition \ref{prop:bounding_mpc_cost}, and $\xi_N$ and $\eta_N$ as in Proposition \ref{prop:energy_decreasing}, we have
    \begin{equation}
    \label{eq:final_performance_bound}
        J^{[\hmuN]}_{\infty}(x) \leq \frac{1 + \alpha_N}{1 - \xi_N - \eta_N}\Vit(x) + \frac{\beta_N}{1 - \xi_N - \eta_N},
    \end{equation}
where $J^{[\hmuN]}_{\infty}(x)$ is defined in \eqref{eq:MPC_infinite_time_performace}.
\end{theorem}
\begin{proof}
    The proof resembles that of \cite[Prop.~2.2]{grune2008infinite}\cliu{, and we provide it for completeness}. By performing a telescopic sum of \eqref{eq:energy_decreasing}, for any $T \in \bN$, we have
    \begin{multline}
    \label{eq:telesum_original}
        (1 - \xi_N - \eta_N)\sum^{T-1}_{t = 0}l\left(\ctx^{[\hmuN]}(t,x), \hmuN\big(\ctx^{[\hmuN]}(t,x)\big)\right) \\
        \leq \hVn(x) - \hVn(\ctx^{[\hmuN]}(T,x)).
    \end{multline}
    Taking $T \to \infty$, using the performance definition in \eqref{eq:MPC_infinite_time_performace} and $0 \leq \hVn(\ctx^{[\hmuN]}(\infty,x))$, \eqref{eq:telesum_original} yields
    \begin{equation}
    \label{eq:telesum_reformed}
        (1 - \xi_N - \eta_N)J^{[\hmuN]}_{\infty}(x) \leq  \hVn(x).
    \end{equation}
    Then, substituting the bound in \eqref{eq:bounding_relation_general} into \eqref{eq:telesum_reformed} leads to
    \begin{equation}
    \label{eq:proof_final_bound}
        (1 - \xi_N - \eta_N)J^{[\hmuN]}_{\infty}(x) \leq \alpha_N \Vit(x) + \beta_N,
    \end{equation}
    which gives the final bound in \eqref{eq:final_performance_bound} after dividing both sides by the constant $1 - \xi_N - \eta_N$.
\end{proof}

Given that $\alpha_N$, $\beta_N$, and $\xi_N$ are all error-consistent, the worst-case performance bound as in \eqref{eq:final_performance_bound} will increase if the modeling error becomes larger (cf. the non-decreasing property of the error-consistent functions in Def.~\ref{def:error_consistent}). Besides, if the model is perfect, the final performance bound will degenerate to
\begin{equation}
\label{eq:degenerated_performance_bound}
    J^{[\muN]}_{\infty}(x) \leq \frac{1}{1 - \eta_N}\Vit(x).
\end{equation}
Note that \eqref{eq:degenerated_performance_bound} is a variant of the bounds given in \cite{grune2008infinite, kohler2021stability} without modeling error.

\section{Numerical Example}
\label{sec:numerical example}
Consider the true linear system of the form \eqref{eq:sys_linear_true} specified by
\begin{equation*}
    A = 
    \begin{bmatrix}
        1 & 0.7 \\
        0.12 & 0.4
    \end{bmatrix},\quad
    B = 
    \begin{bmatrix}
        1 \\
        1.2
    \end{bmatrix},
\end{equation*}
with the input constraint set given in \eqref{eq:cons_u} with $F_u = [10, -10]^\top$. It should be noted that this considered numerical example is challenging to analyze since it is unstable ($\rho(A) = 1.1171 > 1$), meaning that the allowed level of modeling error and the prediction horizon are restricted to ensure $\xi_N + \eta_N < 1$. Now we analyze its infinite-horizon performance based on the theoretical results. For the quadratic objective, the matrices $Q$ and $R$ are given by
\begin{equation*}
    Q = 
    \begin{bmatrix}
        2 & 0 \\
        0 & 2
    \end{bmatrix},\quad
    R = 1.
\end{equation*}
To simplify the simulation, we consider $\delta_A = \delta_B = \delta$, and we investigate the behavior of the derived performance bound when $\delta$ varies between $10^{-3}$ and $10^{-2}$ for a fixed prediction horizon. For a given level of modeling error $\delta \in [10^{-3}, 10^{-2}]$, we simulate $100$ different estimated systems that are generated randomly and satisfy \eqref{eq:uncertainty_set}, and for each of the them, its corresponding feedback gain $K$ as in Lemma \ref{lm:local_stabilization_property} is designed using LQR with the tuple $(\hA, \hB, Q, R)$. The code used for the simulations in this paper is available on GitHub.\footnote{See \url{https://github.com/lcrekko/lq\_mpc}} The behavior of $\alpha_N$, $\beta_N$, $\xi_N$, and $J_{\text{bound}}$ (the bound given in \eqref{eq:final_performance_bound}) as a function of $\delta$ is shown in Fig.~\ref{fig:error_variation}.
\begin{figure}[h]
    \centering
    \includegraphics[width = 0.48\textwidth]{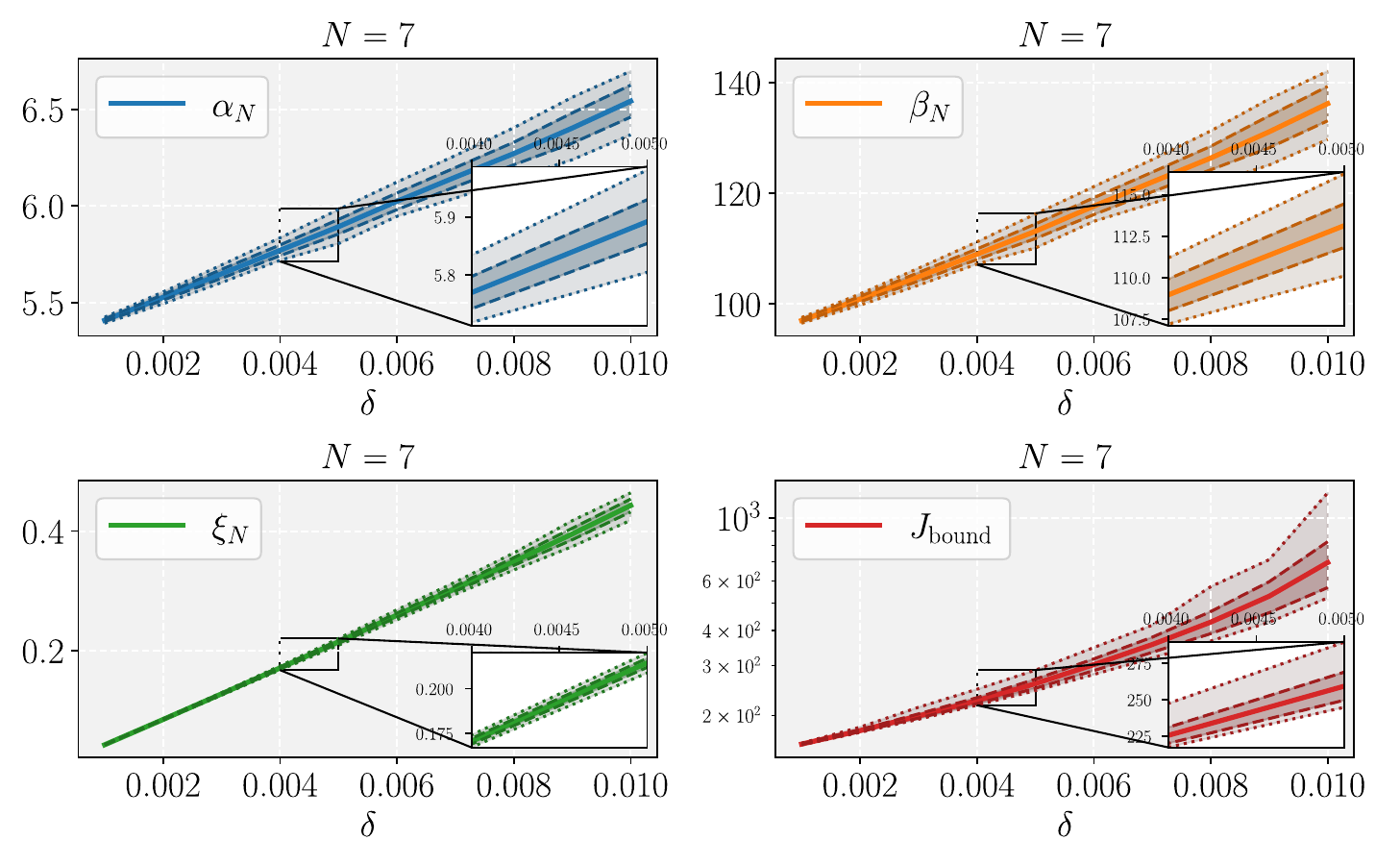}
    \caption{The behavior of $\alpha_N$ (upper left), $\beta_N$ (upper right), $\xi_N$ (lower left), and $J_{\text{bound}}$ (lower right) for a varying modeling error $\delta \in [10^{-3}, 10^{-2}]$. For a given error level $\delta$, $100$ estimated systems are simulated, and the mean, variance, and max (min) of each of the four quantities are shown, respectively, using a solid line, dashed lines, and dotted lines.}
    \label{fig:error_variation}
\end{figure}
The results in Fig. \ref{fig:error_variation} indicate that an increased modeling error leads to a larger difference between the value functions $\widehat{V}_N$ and $V_\infty$ (by having larger values of $\alpha_N$ and $\beta_N$), a mitigated energy-decreasing property (by having a larger value of $\xi_N$), and thus an increased worst-case performance bound. 

On the other hand, given a specific value of $\delta$, we also simulated the behavior of the four quantities $\alpha_N$, $\beta_N$, $\xi_N$, and $J_{\text{bound}}$ when varying the prediction horizon, and the results are presented in Fig. \ref{fig:horizon_variation}. According to Fig. \ref{fig:horizon_variation}, as the prediction horizon is extended, the error accumulates, which is reflected in the increasing behavior of $\alpha_N$, $\beta_N$, and $\xi_N$. However, the performance bound is not necessarily a monotonous function of the prediction horizon since $\eta_N$ decreases as $N$ increases. For $\delta=5\cdot 10^{-3}$, $N = 7$ is the prediction horizon that achieves the smallest worst-case performance bound.
\begin{figure}[h]
    \centering
    \includegraphics[width = 0.48\textwidth]{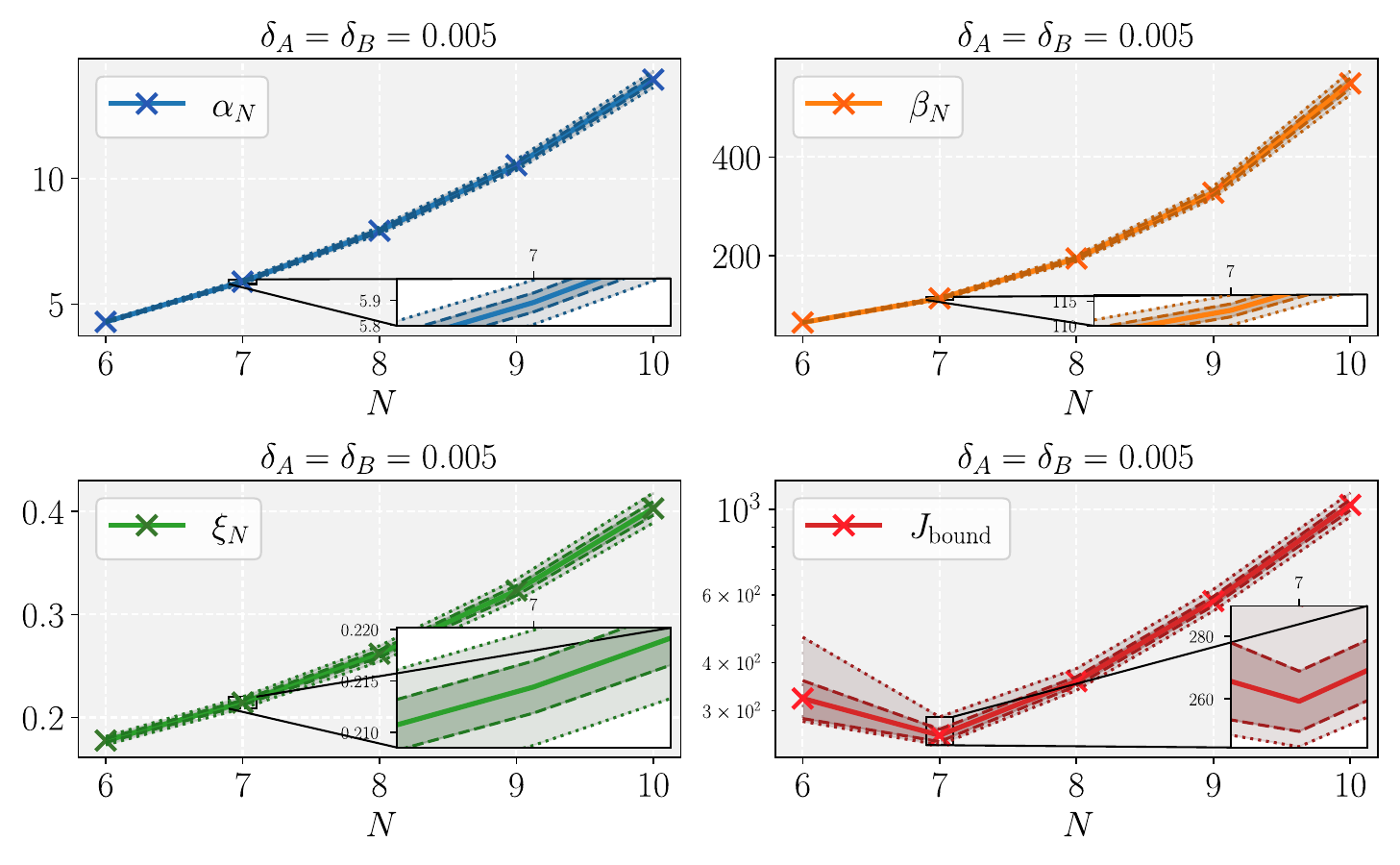}
    \caption{The behavior of $\alpha_N$ (upper left), $\beta_N$ (upper right), $\xi_N$ (lower left), and $J_{\text{bound}}$ (lower right) for a varying prediction horizon $N \in \{6,7,8,9,10\}$. For the specified $\delta = 5\cdot 10^{-3}$, $100$ estimated systems are simulated, and the mean, variance, and max (min) of each of the four quantities are shown, respectively, using a solid line, dashed lines, and dotted lines.}
    \label{fig:horizon_variation}
\end{figure}
For this numerical example, the optimal cost $V_\infty = 0.20229$, and the true closed-loop MPC cost varies between from $0.20229 + 5\cdot10^{-5}$ to $0.20229 + 1.1\cdot10^{-4}$. In Fig. \ref{fig:true_cost}, the true cost of the certainty-equivalent MPC controller is shown, revealing that our performance bound is conservative. This conservatism primarily arises from repeatedly applying various inequalities, leading to the accumulation of relaxation errors. 
In general, RDP-based theoretical performance analysis suffers from conservatism (see also \cite{kohler2023stability, giselsson2013feasibility}) since the RDP inequality only provides a lower bound on the decreased energy of the closed-loop system (cf. \eqref{eq:rdp_prinstine}).
\begin{figure}[h]
    \centering
    \includegraphics[width = 0.48\textwidth]{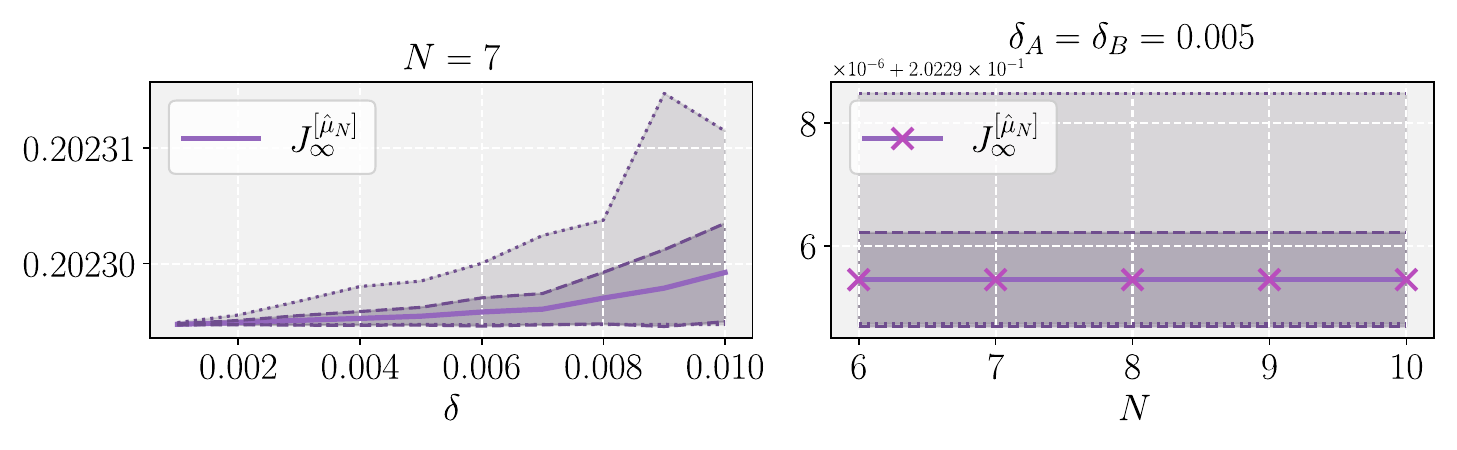}
    \caption{The behavior of the true performance as a function of the modeling error (left) and prediction horizon (right). For a given error level $\delta$, $100$ estimated systems are simulated, and the mean, variance, and max (min) of each of the four quantities are shown, respectively, using a solid line, dashed lines, and dotted lines.}
    \label{fig:true_cost}
\end{figure}
Furthermore, the optimal prediction horizon deduced using the worst-case performance bound is \textit{not} necessarily the optimal prediction horizon $N_{\text{OPT}}$ that achieves the best performance, which is reflected in the right subplot of Fig. \ref{fig:true_cost} where we observe that the true performance does not vary significantly when the horizon changes. This discrepancy is due to that facts i) that $N_{\text{OPT}}$ depends on the linear feedback gain $K$ that is chosen subject to user preference and ii) that the derivation based on the energy-decreasing property in \eqref{eq:rdp_prinstine} is only sufficient but not necessary. However, our theoretical analysis still highlights the tension between choosing a long horizon for optimality and a short horizon for a lower prediction error, and thus the conclusions can still provide design insights for choosing a proper prediction horizon of certainty-equivalent MPC.

\section{Conclusions}
\label{sec:conclusion}
This paper has provided stability and closed-loop performance analyses of MPC for uncertain linear systems. We have derived a performance bound quantifying the suboptimality gap between a certainty-equivalent MPC controller and the ideal infinite-horizon optimal controller with access to the true system model. Additionally, we have established a sufficient condition on the prediction horizon and the model mismatch to ensure the stability of the closed-loop system. Furthermore, our analysis reveals how the prediction horizon and model mismatch jointly influence optimality. These insights offer valuable guidance for designing and implementing MPC controllers for uncertain linear systems in terms of achieving a desired level of identification error and choosing a suitable prediction horizon for guaranteed performance. Potential future directions include extending the existing analysis framework to MPC with terminal costs and analyzing the performance of adaptive MPC and learning-based MPC that learns the model online.






\bibliographystyle{IEEEtran}
\bibliography{4-references/additional_reference,4-references/basic_references,4-references/learning,4-references/network,4-references/oco,4-references/intro_used,4-references/ourpaper}

\useRomanappendicesfalse 
\appendix
\subsection{Matrix Inequalities}
\label{appendix:A--quadratic_norm}

\begin{lemma}
    \label{lm:difference_power_single}
    Given matrices $M_1, M_2 \in \bR^{n\times n}$ and any well-posed matrix norm $\|\cdot\|$ that possesses sub-additivity and sub-multiplicativity, we have
    \begin{equation}
        \label{eq:difference_power_single}
        \|M^i_1 - M^i_2\| \leq (\delta_M + \|M_2\|)^{i} - \|M_2\|^{i},
    \end{equation}
    where $i \in \bN$ and $\|M_1 - M_2\| \leq \delta_M$.
\end{lemma}
\begin{proof}
For $i = 0$, we have $\|M^i_1 - M^i_2\| = \|I_n - I_n\| = 0$ and $(\delta_M + \|M_2\|)^{i} - \|M_2\|^{i} = 1 - 1 = 0$, indicating $\|M^i_1 - M^i_2\| = (\delta_M + \|M_2\|)^{i} - \|M_2\|^{i}$. On the other hand, for $i \geq 1$, we have
    \begin{align*}
        & \|M^i_1 - M^i_2\| \\
        = & \|(M_1 - M_2)(M^{i-1}_1 + M^{i-2}_1M_2 + \cdots + M_2^{i-1})\| \\
        = & \|M_1 - M_2\|\sum^{i-1}_{j=0}\|M^{i-j-1}_1M^j_2\| \\
        = & \|M_1 - M_2\|\sum^{i-1}_{j=0}\|(M_1 - M_2) + M_2\|^{i-j-1}\|M_2\|^j\\
        \leq & \delta_M \sum^{i-1}_{j=0}(\delta_M + \|M_2\|)^{i-j-1}\|M_2\|^j \\
        \leq & \big[(\delta_M +\|M_2\|)-\|M_2\|\big]\sum^{i-1}_{j=0}(\delta_M + \|M_2\|)^{i-j-1}\|M_2\|^j \\
        \leq & (\delta_M +\|M_2\|)^{i} - \|M_2\|^i,
    \end{align*}
which completes the proof.
\end{proof}

\begin{lemma}
    \label{lm:difference_power_double}
    Given matrices $M_1, M_2 \in \bR^{n\times n}$, $N_1, N_2 \in \bR^{n\times m}$, and any well-posed matrix norm $\|\cdot\|$ that possesses sub-additivity and sub-multiplicativity, we have
    \begin{multline}
        \label{eq:difference_power_double}
        \|M^i_1N_1 - M^i_2N_2\| \leq \delta_N \|M_2\|^i + \\ 
        + (\delta_N + \|N_2\|)\big[(\delta_M + \|M_2\|)^{i} - \|M_2\|^{i}\big],
    \end{multline}
    where $i \in \bN$, $\|M_1 - M_2\| \leq \delta_M$, and $\|N_1 - N_2\| \leq \delta_N$.
\end{lemma}
\begin{proof}
    \begin{align*}
         & \|M^i_1N_1 - M^i_2N_2\| \\
       = & \|M^i_1N_1 - M^i_2N_1 + M^i_2N_1 - M^i_2N_2\| \\
       = & \|M^i_2(N_1 - N_2) + (M^i_1 - M^i_2)N_1\| \\
       \leq & \|M_2\|^i\|N_1 - N_2\| + \|M^i_1 - M^i_2\|\|(N_1 - N_2) + N_2\| \\
       \leq & \delta_N\|M_2\|^i + (\delta_N + \|N_2\|)\|M^i_1 - M^i_2\| \\
       \leq & \delta_N \|M_2\|^i + (\delta_N + \|N_2\|)\big[(\delta_M + \|M_2\|)^{i} - \|M_2\|^{i}\big],
    \end{align*}
where the last inequality is due to Lemma \ref{lm:difference_power_single}.
\end{proof}

\begin{lemma}
\label{lm:power_difference}
    Given the matrices $\hA$ and $\hB$ satisfying $\hat{A} \in \cA(A, \eA) \subseteq \bR^{n \times n}$ and $\hat{B} \in \cB(B, \eB) \subseteq \bR^{n \times m}$, where $\cA(A, \eA)$ and $\cB(B, \eB)$ are defined as in \eqref{eq:uncertainty_set}, we have, for $i \in \bN$, 
    \begin{subequations}
        \begin{multline}
        \label{eq:difference_power_A}
             \|A^i - \hA^i\|_2 \leq (\eA + \ntwo{\hA})^{i} - (\ntwo{\hA})^{i},
        \end{multline}
        \begin{multline}
        \label{eq:difference_power_AB}
            \|A^iB - \hA^i\hB\|_2 \leq \eB(\ntwo{\hA})^i  \\
            + (\eB + \ntwo{\hB})\big[(\eA + \ntwo{\hA})^{i} - (\ntwo{\hA})^{i}\big].
        \end{multline}
    \end{subequations}
\end{lemma}
\begin{proof}
    Applying Lemma \ref{lm:difference_power_single} to $A$ and $\hA$ directly leads to \eqref{eq:difference_power_A}. Likewise, \eqref{eq:difference_power_AB} is a direct consequence of applying Lemma \ref{lm:difference_power_double} to the set of matrices $A$, $B$, $\hA$ and $\hB$.
\end{proof}

Based on \eqref{eq:difference_power_A} and \eqref{eq:difference_power_AB}, we define two error-consistent functions $g^{(n)}_{i,(x)}$ and $g^{(n)}_{i,(u)}$ for all $n \in \bN_+$ as
\begin{subequations}
\label{eq:function_g}
    \begin{align}
        & \hspace{-0.2cm} g^{(n)}_{i,(x)}(\eA, \cdot) = \big[(\eA + \ntwo{\hA})^{i} - (\ntwo{\hA})^{i}\big]^n, \\
        & \hspace{-0.2cm} g^{(n)}_{i,(u)}(\eA, \eB) \seq \bigg\{(\eB + \ntwo{\hB})g^{(1)}_{i,(x)}(\eA, \cdot)  \notag \\
        & \hspace*{30ex} + \eB(\ntwo{\hA})^i\bigg\}^n \spd
    \end{align}
\end{subequations}

\begin{lemma}
\label{lm:difference_cross_product}
Given matrices $M_1, M_2 \in \bR^{p\times q}$, $\Xi \in \bR^{q \times n}$, $N_1, N_2 \in \bR^{n \times m}$, and any well-posed matrix norm $\|\cdot\|$ that possesses sub-additivity and sub-multiplicativity, we have
\begin{multline}
    \|M_1\Xi N_1 - M_2 \Xi N_2\| \leq \|\Xi\|\|M_1 - M_2\|\|N_1 - N_2\|  \\
    + \|M_1\Xi\|\|N_1 - N_2\| + \|\Xi N_1\|\|M_1 - M_2\|
\end{multline}
\begin{proof}
\cliu{The proof consists of simple algebraic manipulations, which we present for completeness.}
    \begin{align*}
        & \|M_1\Xi N_1 - M_2 \Xi N_2\| \\
       = &  \|M_1\Xi N_1 - M_1\Xi N_2 + M_1\Xi N_2 - M_2 \Xi N_2\| \\
       \leq & \|M_1\Xi (N_1 - N_2)\| + \|(M_1 - M_2)  \Xi N_2\| \\
       \leq & \|M_1\Xi\|\|N_1 - N_2\| + \|M_1 - M_2\| \|\Xi N_1\| + \\
       & \hspace*{10ex} + \|M_1 - M_2\| \|\Xi (N_2 - N_1)\| \\
       \leq & \|M_1\Xi\|\|N_1 - N_2\| + \|M_1 - M_2\| \|\Xi N_1\| + \\
       & \hspace*{10ex} + \|\Xi \| \|M_1 - M_2\| \|N_1 - N_2\|
    \end{align*}
\end{proof}
\end{lemma}

\begin{corollary}
\label{corollary:difference_cross_product_transpose}
Given matrices $M_1, M_2 \in \bR^{p\times q}$, $\Xi \in \bR^{p \times p}$, and any well-posed matrix norm $\|\cdot\|$ that possesses sub-additivity and sub-multiplicativity, we have
\begin{multline}
    \|M^\top_1\Xi M_1 - M^\top_2 \Xi M_2\| \leq \|\Xi\|\|M_1 - M_2\|^2 + \\
    + (\|M^\top_1\Xi\| + \|\Xi M_1\|)\|M_1 - M_2\|.
\end{multline}
\end{corollary}

\begin{lemma}
    \label{lm:system_level_synthesis_matrix}
Given matrices $A$ and $B$ as in \eqref{eq:sys_linear_true} and their estimated counterparts $\hA$ and $\hB$, we define $\bQN \sdeq I_{N} \otimes Q$, 
\begin{equation*}
    \hspace{-0.25cm}\PhiN \sdeq 
    \begin{bmatrix}
        I \\
        A \\
        A^2 \\
        \vdots \\
        A^N
    \end{bmatrix},
    \GaN \sdeq
    \begin{bmatrix}
        0 & 0 & \cdots & 0 \\
        B & 0 & \cdots & 0 \\
        AB & B & \cdots & 0 \\
        \vdots & \vdots & & \vdots \\
        A^{N-1}B & A^{N-2}B & \cdots & B
    \end{bmatrix}\hspace{-0.15cm}\scm
\end{equation*}
\begin{equation*}
    \hspace{-0.25cm}\hPhiN \sdeq 
    \begin{bmatrix}
        I \\
        \hA \\
        \hA^2 \\
        \vdots \\
        \hA^N
    \end{bmatrix},
    \hGaN \sdeq
    \begin{bmatrix}
        0 & 0 & \cdots & 0 \\
        \hB & 0 & \cdots & 0 \\
        \hA\hB & \hB & \cdots & 0 \\
        \vdots & \vdots & & \vdots \\
        \hA^{N-1}\hB & \hA^{N-2}\hB & \cdots & \hB
    \end{bmatrix}\hspace{-0.15cm}.
\end{equation*}
\cliu{
Then, we have the following two inequalities as
\begin{subequations}
\label{eq:k_and_CapitalK}
    \begin{align}
        \label{eq:capitalK_original_results}
        & \|\hGaN^\top \overline{Q}_{N+1}\hGaN - \GaN^\top \overline{Q}_{N+1}\GaN\|_2 \leq \theta_{N,(u)}, \\
        \label{eq:k_original_results}
        & \|\hGaN^\top \overline{Q}_{N+1}\hPhiN - \GaN^\top \overline{Q}_{N+1}\PhiN\|_2 \leq \theta_{N,(x,u)}
    \end{align}
\end{subequations}
}
where $\theta_{N,(u)}$ and $\theta_{N,(x,u)}$ are given by
\begin{subequations}
\label{eq:function_theta}
    \begin{align}
        & \hspace*{-1ex} \theta_{N,(u)} = \overline{\sigma}_Q\big(2\|\hGaN\|_2\bar{g}_{(u)} + \bar{g}^2_{(u)}\big), \\
        & \hspace*{-1ex} \theta_{N,(x,u)} = \overline{\sigma}_Q\big(\|\hGaN\|_2\bar{g}_{(x)} + \|\hPhiN\|_2\bar{g}_{(u)} + \bar{g}_{(x)}\bar{g}_{(u)}\big),
    \end{align}
\end{subequations}
in which $\bar{g}_{(x)}$ and $\bar{g}_{(u)}$ are, respectively, defined as $\bar{g}_{(x)} := \sum^N_{i=1}g^{(1)}_{i,(x)}$ and $\bar{g}_{(u)} := \sum^{N}_{i=1}\sum^{i-1}_{j=0}g^{(1)}_{j,(u)}$ with $g^{(1)}_{i,(x)}$ and $g^{(1)}_{i,(u)}$ given as in \eqref{eq:function_g}.
\end{lemma}

\begin{proof}
    We first provide the proof of \eqref{eq:k_original_results}. Due to Lemma \ref{lm:difference_cross_product}, we have
    \begin{multline*}
        \|\hGaN^\top \overline{Q}_{N+1}\hPhiN - \GaN^\top \overline{Q}_{N+1}\PhiN\|_2 \leq \\
        \|\overline{Q}_{N+1}\|_2\bigg(\|\hGaN\|_2\|\hPhiN - \PhiN\|_2  \\ +\|\hPhiN\|_2\|\hGaN - \GaN\|_2 + \|\hPhiN - \PhiN\|_2\|\hGaN - \GaN\|_2\bigg).
    \end{multline*}
In addition, due to lemma \ref{lm:power_difference} and the property of the spectral norm of block matrix \cite{horn2012matrix}, we have $\|\hPhiN - \PhiN\|_2 \leq \sum^N_{i=1}\|A^i - \hA^i\|_2 \leq \sum^N_{i=1}g^{(1)}_{i,(x)}$ and $\|\hGaN - \GaN\|_2 \leq \sum^N_{i=1}\sum^{i-1}_{j=0}\|A^jB - \hA^j\hB\|_2 \leq \sum^N_{i=1}\sum^{i-1}_{j=0}g^{(1)}_{j,(u)}$. The final inequality as in \eqref{eq:k_original_results} is thus established by substituting the above results.

The proof of \eqref{eq:capitalK_original_results} follows a similar procedure using Corollary \ref{corollary:difference_cross_product_transpose}.
\end{proof}

\begin{lemma}
    \label{lm:quadratic_norm}
    Given two sequences of vectors $\{a_i\}^{N}_{i=1}$ and $\{b_i\}^{N}_{i=1}$ where $a_i, b_i \in \bR^n, \forall i \in \bZ^+_N$ and a symmetric positive definite matrix $Q \in \bR^{n\times n}$, we have
    \vspace{-0.2cm}
    \begin{multline}
    \label{eq:quadratic_norm}
        \hspace{-0.35cm} \sum^N_{i=1} \nQ{a_i \pm b_i} \leq \sum^N_{i=1}\left[\nQ{a_i} \splus \nQ{b_i} \splus 2(\nQ{a_i}\nQ{b_i})^\frac{1}{2}\right] \\
        \hspace{-0.1cm} \leq \sum^N_{i=1}(\nQ{a_i} \splus \nQ{b_i}) \splus 2\left[(\sum^N_{i=1} \nQ{a_i})(\sum^N_{i=1} \nQ{b_i})\right]^\frac{1}{2} \hspace{-0.1cm}
    \end{multline}
\end{lemma}
\vspace{0.1cm}
\begin{proof}
\cliu{The proof trivially applies the Cauchy–Schwarz inequality, and we provide it here for completeness.} We denote the Cholesky decomposition of the matrix $Q$ by $Q = \Gamma^\top_Q \Gamma_Q$, then we can proceed as follows:
\begin{align*}
    & \sum^N_{i=1}\nQ{a_i \pm b_i} \\ 
    \leq & \sum^N_{i=1}\left[\nQ{a_i} \splus \nQ{b_i} \splus 2|a_i^\top Qb_i|\right] \\
    \leq & \sum^N_{i=1}\left[\nQ{a_i} \splus \nQ{b_i} \splus 2(\|\Gamma_Q a_i\|^2_2\|\Gamma_Q b_i\|^2_2)^{\frac{1}{2}}\right] \\
    \leq & \sum^N_{i=1}\left[\nQ{a_i} \splus \nQ{b_i} \splus 2(\nQ{a_i}\nQ{b_i})^{\frac{1}{2}}\right],
\end{align*}
where the second inequality holds due to the Cauchy–Schwarz inequality. We thus proved the first inequality in \eqref{eq:quadratic_norm}, and further applying Cauchy–Schwarz inequality to the product $(\nQ{a_i}\nQ{b_i})^{\frac{1}{2}}$ in the above result leads to
\begin{align*}
    &\sum^N_{i=1}\left[\nQ{a_i} \splus \nQ{b_i} \splus 2(\nQ{a_i}\nQ{b_i})^{\frac{1}{2}}\right] \\
    \leq & \sum^N_{i=1}(\nQ{a_i} \splus \nQ{b_i}) \splus 2\left[(\sum^N_{i=1} \nQ{a_i})(\sum^N_{i=1} \nQ{b_i})\right]^\frac{1}{2},
\end{align*}
which is the second inequality in \eqref{eq:quadratic_norm}.
\end{proof}

\subsection{Technical Lemmas}
\label{appendix:B--Lemma_exponential decay}
\subsubsection{Proof of Lemma \ref{lm:local_stabilization_property}}
Given a stabilizable matrix pair $(A, B)$, by definition there exists a matrix $K$ such that $\rho(\Acl) := \rho(A+BK) < 1$. Due to Gelfand's formula \cite[Lemma IX.1.8]{dunford1988linear}, we know that there exist scalars $\lambda_K \geq 1$ and $\rho_K \in (0,1)$ s.t. $\|(\Acl)^k\|_2 \leq \lambda_K (\sqrt{\rho_K})^k$, $\forall k \in \bN$. On the other hand, for any vector $x \in \bR^n$, we have
\begin{align*}
    \|x\|^2_{Q+K^\top RK} &= \|x\|^2_{Q} + \|Kx\|^2_{R} \leq \|x\|^2_{Q} + \overline{\sigma}_R \|Kx\|^2_2 \\
    &\leq \|x\|^2_{Q} + \overline{\sigma}_R\|K\|^2_2 \|x\|^2_2 \\
    &\leq (1 + \underline{\sigma}^{-1}_Q\overline{\sigma}_R\|K\|^2_2)\|x\|^2_Q.
\end{align*}
Thus, $Q + K^\top RK \preceq (1 + \underline{\sigma}^{-1}_Q\overline{\sigma}_R\|K\|^2_2)Q$. As such, for $k \in \bN_+$ and for the linear control law $u = \kappa(x) = Kx$, the closed-loop stage cost can be bounded as
\begin{align*}
    &\hspace*{-5ex}l(\phi_x^{[\kappa]}(k,x), K\phi_x^{[\kappa]}(k,x)) \\
    = &x^\top \left\{[(\Acl)^k]^\top (Q + K^\top RK) [(\Acl)^k]\right\} x \\
    \leq &(1 + \underline{\sigma}^{-1}_Q\overline{\sigma}_R\|K\|^2_2)r_Q\|(\Acl)^k\|^2_2\|x\|^2_Q \\
    \leq &(1 + \underline{\sigma}^{-1}_Q\overline{\sigma}_R\|K\|^2_2)r_Q(\lambda_K)^2 (\rho_K)^k\nQ{x}.
\end{align*}
Taking $C^\ast_K = \big(1+\underline{\sigma}^{-1}_Q\overline{\sigma}_R\|K\|^2_2\big)r_Q\lambda_{K}^2$, we arrive at \eqref{eq:local_exponential} by combining the above results.

\subsubsection{Linear bound on square root function} We provide a trivial lemma about the square root function.
\begin{lemma}
\label{lm:square_root}
Given $x \in \bR_+$, we have $\sqrt{x} \leq px + q$ \cliu{for any pairs $(p, q)$ such that $4pq = 1$ and $p > 0$.}
\end{lemma}
\begin{proof}
    By the simple AM-GM inequality, we know $px + q \geq 2(pqx)^{\frac{1}{2}} = \sqrt{x}$.
\end{proof}

\subsubsection{Difference between the optimal inputs} As in Section \ref{subsec:mpc_value_function}, we define the difference between the two optimal inputs as $\du(x) := \un(x) - \hun(x)$, the following lemma provides an upper bound on $\|\du(x)\|_2$. It is worth noting that the obtained upper bound as in \eqref{eq:input_upper_bound_original} is closely related to the term $\Eu$ as in \eqref{eq:Eu}. \cliu{To establish the desired bound, we first present an intermediate result on the sensitivity analysis of QPs. Specifically, we provide the following Lemma \ref{lm:sensitivity_qp_less_conservative} in which an upper bound on the distance between the optimal solutions of two parametric QPs is provided. The derived bound, as a slight modification compared to that in \cite[Theorem 2.1]{daniel1973stability}, is less conservative and more suitable to our performance analysis.
\begin{lemma}
    \label{lm:sensitivity_qp_less_conservative}
    Given symmetric positive-definite matrices $\Xi$ and $\hXi$ and a nonempty constraint set $\mathcal{C}_x = \{x \in \mathbb{R}^{n_x} \mid Gx \leq g\}$, we formulate an original QP as
    \begin{equation*}
        \mathrm{P}_{\mathrm{QP}}: \min \left\{\frac{1}{2}x^\top \Xi x + x^\top \zeta \right\} \; \text{s.t. } x \in \mathcal{C}_x,
    \end{equation*}
    and its \textit{perturbed} counterpart as
    \begin{equation*}
        \mathrm{P}_{\mathrm{QP-PT}}: \min \left\{\frac{1}{2}x^\top \hXi x + x^\top \hzeta \right\} \; \text{s.t. } x \in \mathcal{C}_x.
    \end{equation*}
    Denote the optimal solution to the problem $\mathrm{P}_{\mathrm{QP}}$ and that to the problem $\mathrm{P}_{\mathrm{QP-PT}}$, respectively, by $x^\ast$ and $\widehat{x}^\ast$, then we have
    \begin{multline}
        \|x^\ast - \widehat{x}^\ast\|_2 \leq \min\bigg\{\bar{d}_x(G,g), \\ \frac{1}{\underline{\sigma}_{\hXi}}\left(\|x^\ast\|_2\|\Xi-\hXi\|_2 + \|\zeta -\hzeta\|_2 \right)\bigg\},
    \end{multline}
    where the function $\bar{d}_x(G,g) := \max_{x_1, x_2 \in \mathcal{C}_x}\|x_1 - x_2\|_2$.    
\end{lemma}
\begin{proof}
    The proof follows that in \cite[Theorem 2.1]{daniel1973stability} and is omitted here.
\end{proof}
Following the above sensitivity analysis, the upper bound on $\|\du(x)\|_2$ can be established in the following lemma.}
\cliu{
\begin{lemma}
    \label{lm:input_difference_bound}
    The input difference $\du(x)$ satisfies
    \begin{equation}
    \label{eq:input_upper_bound_original}
        \|\du(x)\|_2 \leq \Delta_{N, (\du)}(x),
    \end{equation}
    where $\Delta_{N, (\du)}(x)$ equals
    \begin{equation}
    \label{eq:details_input_upper_bound_original}
        \hspace*{-0.25cm} \min\bigg\{(N\overline{d}_u)^{\frac{1}{2}}, \\ \frac{1}{\underline{\sigma}_{\hHN}}\left((N\overline{u})^{\frac{1}{2}}\theta_{N,(u)} \splus \|x\|_2\theta_{N,(x,u)} \right)\bigg\},\hspace*{-0.15cm}
    \end{equation}
    in which $\theta_{N,(x)}$ and $\theta_{N,(x,u)}$ are defined as in \eqref{eq:function_theta}, $\hHN = \bRN+\hGaN^\top\bQNp\hGaN$ with $\bRN = I_{N} \otimes R$ and $\bQNp = I_{N+1} \otimes Q$, $\overline{d}_{u} := \max_{u_1,u_2 \in \cU}\ntwo{u_1 - u_2}^2$, and $\overline{u} := \max_{u \in \cU}\ntwo{u}^2$.                        
\end{lemma}
}
\begin{proof}
    For the problem $\mathrm{P}_{\mathrm{ID-MPC}}$, we form decision variable vectors $\mathbf{x}_{N,t} = [x_{0|t}^\top, x_{1|t}^\top, \dots, x_{N|t}^\top]^\top$ and $\mathbf{u}_{N-1,t} = [u_{0|t}^\top, u_{1|t}^\top, \dots, u_{N-1|t}^\top]^\top$. We exclude $u_{N|t}$ since its optimal value is $0$, which does not influence the following reasoning. By \eqref{eq:state_evolution}, we obtain
    \begin{equation*}
        \mathbf{x}_{N,t} = \PhiN x + \GaN \mathbf{u}_{N-1,t},
    \end{equation*}
    where $\PhiN$ and $\GaN$ are defined as in Lemma \ref{lm:system_level_synthesis_matrix} \cliu{and we assign $x = x(t)$ for notational simplicity w.l.o.g. .} By noting that the objective function of $\mathrm{P}_{\mathrm{ID-MPC}}$ can be rewritten as $\mathbf{x}^\top_{N,t}\overline{Q}_{N+1}\mathbf{x}_{N,t} + \mathbf{u}^\top_{N-1,t}\overline{R}_N\mathbf{u}_{N-1,t}$, an QP reformulation of the problem $\mathrm{P}_{\mathrm{ID-MPC}}$ can be obtained as
        \begin{align}
        \mathrm{QP}_{\mathrm{ID-MPC}}: & \min_{\mathbf{u}_{N-1,t}} \frac{1}{2}\mathbf{u}^\top_{N-1,t}H_N\mathbf{u}_{N-1,t} + \mathbf{u}^\top_{N-1,t}b_N
        \notag
        \end{align}
        \vspace{-0.35cm}
        \begin{align}
        \text{s.t.} \; &\; (I_N\otimes F_u)\mathbf{u}_{N-1,t} \leq \mathds{1}_{Nc_u}, \notag
        \end{align}
    where $H_N = \bRN+\GaN^\top\bQNp\GaN$ and $b_N = \GaN^\top\overline{Q}_{N+1}\PhiN x$. Similarly, we reformulate the problem $\mathrm{P}_{\mathrm{NM-MPC}}$ as
        \begin{align}
        \mathrm{QP}_{\mathrm{NM-MPC}}: & \min_{\mathbf{u}_{N-1,t}} \frac{1}{2}\mathbf{u}^\top_{N-1,t}\hHN\mathbf{u}_{N-1,t} + \mathbf{u}^\top_{N-1,t}\hat{b}_N
        \notag
        \end{align}
        \vspace{-0.35cm}
        \begin{align}
        \text{s.t.} \; &\; (I_N\otimes F_u)\mathbf{u}_{N-1,t} \leq \mathds{1}_{Nc_u}, \notag
        \end{align}
    where $\hat{b}_N = \hGaN^\top\overline{Q}_{N+1}\hPhiN x$. By treating $\mathrm{QP}_{\mathrm{ID-MPC}}$ as the original optimization problem, $\mathrm{QP}_{\mathrm{NM-MPC}}$ can be viewed as its perturbed counterpart. \cliu{Therefore, leveraging the bound given in Lemma \ref{lm:sensitivity_qp_less_conservative}, we have
    \begin{multline}
    \label{eq:input_bound_initial}
        \|\du(x)\|_2 \leq \min\bigg\{\overline{d}_{\mathbf{u}_{N-1,t}}(I_N\otimes F_u, \mathds{1}_{Nc_u}), \\ \frac{1}{\underline{\sigma}_{\hHN}}\left(\|\un(x)\|_2\|H_N - \hHN\|_2 + \|b_N - \hat{b}_N\|_2 \right)\bigg\}.
    \end{multline}}
    Due to \eqref{eq:capitalK_original_results} and \eqref{eq:k_original_results}, we know that
    \begin{subequations}
    \label{eq:transition_matrix_difference_bound}
        \begin{align}
            & \|H_N - \hHN\|_2 = \|\GaN^\top\bQNp\GaN - \hGaN^\top\bQNp\hGaN\|_2 \notag \\
            & \hspace*{30ex}  \leq \theta_{N,(u)}, \\
            & \|b_N - \hat{b}_N\|_2 = \|\GaN^\top\overline{Q}_{N+1}\PhiN x - \hGaN^\top\overline{Q}_{N+1}\hPhiN x\|_2 \notag \\
            & \hspace*{30ex} \leq \|x\|_2 \theta_{N,(x,u)}.
        \end{align}
    \end{subequations}
    By substituting \eqref{eq:transition_matrix_difference_bound} into \eqref{eq:input_bound_initial} and by noting the facts that $\|\un(x)\|_2 \leq (N\overline{u})^{\frac{1}{2}}$ and $\overline{d}_{\mathbf{u}_{N-1,t}}(I_N\otimes F_u, \mathds{1}_{Nc_u}) = (N\overline{d}_u)^{\frac{1}{2}}$, we can finally obtain the bound given as in \eqref{eq:input_upper_bound_original} and \eqref{eq:details_input_upper_bound_original}.
\end{proof}

\subsubsection{State evolution with zero initial state} The following lemma provides a bound about the open-loop state trajectory starting from $x = 0$ \cliu{in terms of a bound on the corresponding input sequence, and it is useful in establishing the term $\Epsiu$ as in \eqref{eq:Epsiu}}.
\begin{lemma}
\label{lm:zero_initial_propagation}
    Given an admissible input sequence $\mathbf{u}_{N-1}$ such that $\mathbf{u}_{N-1}[i] \in \cU, i = 0,1,\dots,N-1$, the open-loop state that starts from $x = 0$ is given by $\otx(k,0,\mathbf{u}_{N-1}) = \sum^{k-1}_{i=0}A^{k-1-i}B^i\mathbf{u}_{N-1}[i]$. Denote $\bm{\psi}_{N, x}(0,\mathbf{u}_{N-1}) = [\otx^\top(0,0,\mathbf{u}_{N-1}), \otx^\top(1,0,\mathbf{u}_{N-1}), \dots, \otx^\top(N,0,\mathbf{u}_{N-1})]^\top$, we have
    \begin{equation}
        \|\bm{\psi}_{N, x}(0,\mathbf{u}_{N-1})\|_2 \leq \left(\|\hGaN\|_2 + \bar{g}_{(u)}\right)\|\mathbf{u}_{N-1}\|_2,
    \end{equation}
    where $\hGaN$ and $\bar{g}_{(u)}$ are defined as in Lemma \ref{lm:system_level_synthesis_matrix}.
\end{lemma}
\begin{proof}
    Similar to the proof of Lemma \ref{lm:input_difference_bound}, we have $\bm{\psi}_x(0,\mathbf{u}_{N-1}) = \GaN \mathbf{u}_{N-1}$. Then, we can proceed as
    \begin{align*}
        \|\bm{\psi}_{N, x}(0,\mathbf{u}_{N-1})\|_2 & = \|\GaN \mathbf{u}_{N-1}\|_2 \\
        & \leq \|\GaN\|_2 \|\mathbf{u}_{N-1}\|_2 \\
        & \leq \left(\|\hGaN\|_2 + \|\GaN - \hGaN\|_2\right) \|\mathbf{u}_{N-1}\|_2 \\
        & \leq \left(\|\hGaN\|_2 + \bar{g}_{(u)}\right)\|\mathbf{u}_{N-1}\|_2,
    \end{align*}
which completes the proof.
\end{proof}

\subsubsection{One-step prediction error} 
\label{subapp:one_step_error}
As in the proof sketch in Section \ref{subsec:stability_analysis}, we define the one-step prediction error as $\Delta x := (A - \hA)x + (B - \hB)\hmuN(x)$. The following lemma provides a bound on $\Delta x$ in terms of $l(x, \hmuN(x))$.
\begin{lemma}
    \label{lm:prediction_error_bound}
    The one-step prediction error $\Delta x$ satisfies
    \begin{equation}
    \label{eq:prediction_error_bound_h_function}
        \|\Delta x\|^2_2 \leq h(\eA, \eB)l(x, \hmuN(x)),
    \end{equation}
    where $h(\eA, \eB) = \underline{\sigma}^{-1}_Q\eA^2 + \underline{\sigma}^{-1}_R\eB^2$ is error-consistent.
\end{lemma}
\vspace{0.1cm}
\begin{proof}
By the Cauchy-Schwarz and Young's inequalities, when $\eA,\eB > 0$, we can proceed as
    \begin{align*}
        \|\Delta x\|^2_2 &= \ntwo{(A - \hA)x + (B - \hB)\hmuN(x)}^2 \\
        &\leq \left(1 + \myratio \right)\|(A - \hA)x\|^2_2  \\ & \hspace*{5ex} + \left(1 + \invmyratio \right)\|(B - \hB)\hmuN(x)\|^2_2 \\
        &\leq \left(\eA^2 + \frac{ \eB^2\underline{\sigma}_Q }{ \underline{\sigma}_R }\right)\|x\|^2_2  \\
        & \hspace*{5ex} + \left(\eB^2 + \frac{ \eA^2\underline{\sigma}_R }{ \underline{\sigma}_Q }\right)\|\hmuN(x)\|^2_2 \\
        &\leq \left(\frac{\eA^2}{\underline{\sigma}_Q  }+ \frac{\eB^2}{\underline{\sigma}_R } \right)\left( \|x\|^2_Q +  \|\hmuN(x)\|^2_R\right) \\
        &\leq \left(\frac{\eA^2}{\underline{\sigma}_Q  }+ \frac{\eB^2}{\underline{\sigma}_R } \right) l(x, \hmuN(x)).
    \end{align*}
When $\eA = 0$ and $\eB > 0$, we have
\begin{multline*}
    \|\Delta x\|^2_2 = \|(B - \hB)\hmuN(x)\|^2_2
    \\ \leq \frac{\eB^2}{\underline{\sigma}_R}\|\hmuN(x)\|^2_R
    \leq \frac{\eB^2}{\underline{\sigma}_R}l(x, \hmuN(x)).
\end{multline*}
Likewise, when $\eB = 0$ and $\eA > 0$, we have
\begin{multline*}
    \|\Delta x\|^2_2 = \|(A - \hA)x\|^2_2 \\
    \leq \frac{\eA^2}{\underline{\sigma}_Q}\|x\|^2_Q \leq \frac{\eA^2}{\underline{\sigma}_Q}l(x, \hmuN(x)).
\end{multline*}
Finally, when $\eA = \eB = 0$, we have $\|\Delta x\|^2_2 = 0$ and $h(\eA, \eB) = 0$, meaning \eqref{eq:prediction_error_bound_h_function} is also valid. The proof is completed.
\end{proof}

\subsubsection{Multi-step prediction error} 
\label{subapp:part_multistep}
As in Step 3 of the proof sketch in Section \ref{subsec:mpc_value_function}, the $k$-step open-loop prediction error under control input $\hun(x)$) is defined as $e_{\psi}(k) := \otx(k,x,\hun(x)) - \oex(k,x,\hun(x))$, the following lemma, provides an upper bound for $e_{\psi}(k)$.
\cliu{
\begin{lemma}
    \label{lm:multi_step_prediction_error_bound}
    The concatenated prediction error $\mathbf{e}_{\psi,N} = [e^\top_{\psi}(0), e^\top_{\psi}(1), \dots, e^\top_{\psi}(N)]^\top$ satisfies
    \begin{equation}
        \label{eq:multi_step_prediction_error_bound}
        \|\mathbf{e}_{\psi,N}\|^2_2 \leq \Delta_{N, (\psi)}(x),
    \end{equation}
    where $\Delta_{N, (\psi)}(x)$ equals
    \begin{equation}
    \label{eq:detail_multi_step_prediction_error_bound}
         \sum^{N}_{k = 0}\left[\left(g^{(2)}_{k,(x)}\splus\sum^{k-1}_{i=0}g^{(2)}_{k-i-1,(u)}\right)\hspace*{-1ex}\left(\|x\|^2_2 \splus k\overline{u}\right)\right],
    \end{equation}
    in which $g^{(2)}_{k,(x)}$ and $g^{(2)}_{k-i-1,(u)}$ are defined as in \eqref{eq:function_g}.
\end{lemma}
}
\begin{proof}
    For the prediction error $e_{\psi}(k)$, we have
    \begin{multline*}
        e_{\psi}(k) = \underbrace{(A^k - \hA^k)}_{:= D_0}x  \\ 
        + \sum^{k-1}_{i=0}\underbrace{(A^{k-i-1}B - \hA^{k-i-1}\hB)}_{:= D_{i+1}}(\hun(x))[i].
    \end{multline*}
    When $\|D_i\| > 0, i = 0,1,\dots,k$, we select $\epsilon_{i,j} = \|D_j\|^2_2\|D_i\|^{-2}_2$. By the Cauchy-Schwarz and Young's inequalities, we can obtain
    \begin{align*}
        \|e_{\psi}(k)\|^2_2 &= \|D_0x + \sum^{k-1}_{i=0}D_{i+1}(\hun(x))[i]\|^2_2 \\
        &\leq \left(1 + \sum^k_{j=1}\epsilon_{0,j}\right)\|D_0x\|^2_2  \\
        &\hspace*{3ex} + \sum^{k-1}_{i=0}\left(1 + \sum_{\substack{j\neq i+1 \\ 0\leq j\leq k}}\epsilon_{i+1,j}\right)\|D_{i+1}(\hun(x))[i]\|^2_2 \\
        &\leq \left(\sum^k_{i=0}\|D_i\|^2_2\right)\left(\|x\|^2_2 + k\overline{u}\right) \\
        &\leq \left(g^{(2)}_{k,(x)}+\sum^{k-1}_{i=0}g^{(2)}_{k-i-1,(u)}\right)\left(\|x\|^2_2 + k\overline{u}\right).
    \end{align*}
    In cases where $\|D_i\|_2 = 0$ for some $i$'s, we know $D_i = 0$, then proceeding the above procedure without those terms with $D_i = 0$ will reach the same upper bound. The final upper bound as in \eqref{eq:multi_step_prediction_error_bound} is established by summing up the above result over the index $k$.
\end{proof}

\setcounter{proposition}{0}
\subsection{More on Proposition \ref{prop:bounding_mpc_cost}}
\label{appendix:C--bounding_mpc}
\begin{proposition}
    There exist two error-consistent functions $\alpha_N(\eA, \eB)$ and $\beta_N(\eA, \eB)$ such that, for all $x \in \xroa$, $\hVn$ and $\Vit$ satisfy the following inequality:
\begin{equation*} 
\hVn (x) \leq \big (1 + \alpha_N(\eA, \eB)\big) \Vit(x) + \beta_N(\eA, \eB),
\end{equation*}
where functions $\alpha_N$ and $\beta_N$ relate to the eigenvalues and matrix norms of $\hA$, $\hB$, $Q$, $R$, the input constraint set $\cU$, but not to quantities derived from $A$ or $B$. The detailed expressions of $\alpha_N$ and $\beta_N$ are
\begin{subequations}
\label{eq:alpha_and_beta}
    \begin{align}
    \alpha_N &= \max\bigg\{ p_1\Epsi^{\frac{1}{2}} + p_3\Epsiu^{\frac{1}{2}}  \notag \\
    &\quad + p_1p_3(\Epsi\Epsiu)^{\frac{1}{2}}, \;p_2\Eu^{\frac{1}{2}} \bigg\} \\
    \beta_N &= (1+p_1\Epsi^{\frac{1}{2}})(q_3\Epsiu^{\frac{1}{2}} + \Epsiu)  \notag \\
    &\quad  \splus q_2\Eu^{\frac{1}{2}} \splus \Eu + q_1\Epsi^{\frac{1}{2}} \splus \Epsi^{\frac{1}{2}},
    \end{align}
\end{subequations}
where the pairs $(p_i, q_i) \in \bR^2_+$ satisfy $p_iq_i = 1$ and the details of the terms $\Epsi$, $\Eu$, and $\Epsiu$ are given, respectively, in \eqref{eq:Epsi}, \eqref{eq:Eu}, and \eqref{eq:Epsiu}.
\end{proposition}
\vspace{0.1cm}
\cliu{
\noindent \textbf{(1) Details of $\Epsi$}:
Following Lemma \ref{lm:multi_step_prediction_error_bound}, the explicit expression of $\Epsi$ is
\begin{equation}
\label{eq:Epsi}
     \Epsi \seq \overline{\sigma}_Q \Delta_{N, (\psi)}(x),
\end{equation}
where $\Delta_{N, (\psi)}(x)$ is defined as in \eqref{eq:detail_multi_step_prediction_error_bound}.
}

\noindent \textbf{(2) Details of $\Eu$ and $\Epsiu$}:
Based on Lemma \ref{lm:input_difference_bound} and Lemma \ref{lm:zero_initial_propagation}, the explicit form of $\Eu$ is
\cliu{
\begin{equation}
\label{eq:Eu}
    \Eu = \overline{\sigma}_R\left(\Delta_{N, (\du)}(x)\right)^2,
\end{equation}
where $\Delta_{N, (\du)}(x)$ is given as in \eqref{eq:details_input_upper_bound_original}, }and the expression of $\Epsiu$ follows as
\begin{equation}
\label{eq:Epsiu}
    \Epsiu \seq \frac{\overline{\sigma}_Q}{\overline{\sigma}_R}\left(\|\hGaN\|_2 + \bar{g}_{(u)}\right)^2\Eu.
\end{equation}
\begin{proof}
    By definition, $\hVn(x)$ can be expanded as
    \begin{equation}
    \label{eq:proof_pp1_expansion}
        \hspace*{-0.2cm}\hVn(x) \seq \sum^{N}_{k = 0} \left( \nQ{\oex(k, x, \hun(x))} \splus \nR{\big(\hun(x)\big)[k]} \right),
    \end{equation}
    which, by incorporating the prediction error $e_{\psi}(k)$ \cliu{defined as in Appendix \ref{subapp:part_multistep}}, can be rewritten as 
    \begin{multline}
    \label{eq:proof_pp1_prediction_error}
        \hVn(x) = \sum^{N}_{k = 0} \big(\nQ{\otx(k,x,\hun(x)) - e_{\psi}(k)}  \\
        + \nR{\big(\hun(x)\big)[k]} \big).
    \end{multline}
    Applying Lemma \ref{lm:quadratic_norm} to \eqref{eq:proof_pp1_prediction_error}, we obtain
    \begin{multline}
        \hVn(x) \leq \sum^{N}_{k = 0}\left(\nQ{\otx(k, x, \hun(x))} + \nR{\big(\hun(x)\big)[k]} \right)  \\
        + 2\left[\left(\sum^{N}_{k = 0}\nQ{\otx(k,x,\hun(x))}\right) \left(\sum^{N}_{k = 0}\nQ{e_{\psi}(k)} \right) \right]^{\frac{1}{2}}  \\
        \hspace*{27ex} + \sum^{N}_{k = 0}\nQ{e_{\psi}(k)} \\
        \leq \sum^{N}_{k = 0}\left(\nQ{\otx(k, x, \hun(x))} + \nR{\big(\hun(x)\big)[k]} \right)  \\
        + 2\left[\left(\sum^{N}_{k = 0}\nQ{\otx(k,x,\hun(x))}\right) \left(\upQ \|\mathbf{e}_{\psi,N}\|^2_2 \right) \right]^{\frac{1}{2}} + \\
        + \upQ \|\mathbf{e}_{\psi,N}\|^2_2,
    \end{multline}
    where $\mathbf{e}_{\psi,N} = [e^\top_{\psi}(0), e^\top_{\psi}(1), \dots, e^\top_{\psi}(N)]^\top$. Using Lemma \ref{lm:multi_step_prediction_error_bound} and recalling the notation in \eqref{eq:Epsi}, we have
    \begin{multline}
        \hVn(x) \leq \sum^{N}_{k = 0}\bigg(\nQ{\otx(k, x, \hun(x))} + \nR{\big(\hun(x)\big)[k]} \bigg)  \\
        + 2\left[\left(\sum^{N}_{k = 0}\nQ{\otx(k,x,\hun(x))}\right) \Epsi\right]^{\frac{1}{2}} + \\
        + \Epsi.
    \end{multline}
    Then, due to Lemma \ref{lm:square_root}, we can further obtain
    \begin{multline}
        \hVn(x) \leq \left(1 + p_1\Epsi^{\frac{1}{2}}\right)\sum^{N}_{k = 0}\nQ{\otx(k, x, \hun(x))}  \\ 
        + \sum^{N}_{k = 0}\nR{\big(\hun(x)\big)[k]} + q_1\Epsi^{\frac{1}{2}} + \Epsi,
    \end{multline}
    where $p_1q_1 = 1, p_1 > 0$.

    Next, by leveraging the input difference $\big(\du(x)\big)[k] = \big(\un(x)\big)[k] - \big(\hun(x)\big)[k]$, we have
    \begin{multline}
    \label{eq:proof_pp1_cornerstone}
        \hVn(x) \leq \left(1 + p_1\Epsi^{\frac{1}{2}}\right)\sum^{N}_{k = 0}\nQ{\otx(k, x, \hun(x))}  \\ 
        + \sum^{N}_{k = 0}\nR{\big(\un(x)\big)[k] - \big(\du(x)\big)[k]}  \\
        + q_1\Epsi^{\frac{1}{2}} + \Epsi,
    \end{multline}
    In \eqref{eq:proof_pp1_cornerstone}, the term $\sum^{N}_{k = 0}\nR{\big(\un(x)\big)[k] - \big(\du(x)\big)[k]}$ can be proceeded as
    \begin{align}
    \label{eq:proof_pp1_input_difference_derivation}
        & \sum^{N}_{k = 0}\nR{\big(\un(x)\big)[k] - \big(\du(x)\big)[k]} \notag \\
        \leq & \sum^{N}_{k = 0}\bigg(\nR{\big(\un(x)\big)[k]} + \nR{\big(\du(x)\big)[k]}\bigg)  \notag \\
        & \hspace*{1ex} + 2\left[\left(\sum^{N}_{k = 0}\nR{\big(\un(x)\big)[k]}\right) \left(\sum^{N}_{k = 0}\nR{\big(\du(x)\big)[k]} \right) \right]^{\frac{1}{2}} \notag \\
        \leq & (1 + p_2\Eu^{\frac{1}{2}})\sum^{N}_{k = 0}\nR{\big(\un(x)\big)[k]} + \notag \\
        & \hspace*{20ex} + q_2\Eu^{\frac{1}{2}} + \Eu,
    \end{align}
    where the first inequality is due to Lemma \ref{lm:quadratic_norm} \cliu{with $p_2q_2 = 1 (p_2 > 0)$} and the second inequality is due to Lemma \ref{lm:square_root}, Lemma \ref{lm:input_difference_bound}, and the definition as in \eqref{eq:Eu}.
    In addition, due to linearity, we have $\otx(k, x, \hun(x)) = \otx(k, x, \un(x)) - \otx(\cdot,0,\du(x))$, and the term $\sum^{N}_{k = 0}\nQ{\otx(k, x, \hun(x))}$ in \eqref{eq:proof_pp1_cornerstone} can be further proceeded as
    \par\nobreak 
    \vspace{-0.3cm}
    {\small\begin{align}
    \label{eq:proof_pp1_linearity_derivation}
        & \sum^{N}_{k = 0}\nQ{\otx(k, x, \hun(x))} \notag \\
        = & \sum^{N}_{k = 0}\nQ{\otx(k, x, \un(x)) - \otx(k,0,\du(x))} \notag \\
        \leq & \sum^{N}_{k = 0}\bigg(\nQ{\otx(k, x, \un(x))} + \nQ{\otx(k,0,\du(x))}\bigg)  \notag \\
        & \hspace*{-2ex} + 2\left[\left(\sum^{N}_{k = 0}\nQ{\otx(k, x, \un(x))}\right) \left(\sum^{N}_{k = 0}\nQ{\otx(k,0,\du(x))} \right) \right]^{\frac{1}{2}} \notag \\
        \leq & (1 + p_3\Epsiu^{\frac{1}{2}})\sum^{N}_{k = 0}\nQ{\otx(k, x, \un(x))}  \notag \\
        & \hspace*{20ex} + q_3\Epsiu^{\frac{1}{2}} + \Epsiu,
    \end{align}}%
    where the first inequality is due to Lemma \ref{lm:quadratic_norm} \cliu{with $p_3q_3 = 1 (p_3 > 0)$} and the second inequality is due to Lemma \ref{lm:square_root}, Lemma \ref{lm:zero_initial_propagation}, and the definition as in \eqref{eq:Epsiu}.

    Finally, substituting \eqref{eq:proof_pp1_input_difference_derivation} and \eqref{eq:proof_pp1_linearity_derivation} into \eqref{eq:proof_pp1_cornerstone} leads to
    \begin{multline*}
        \hVn(x) \leq (1 + \alpha_N)\Vn(x) + \beta_N \leq (1 + \alpha_N)\Vit(x) + \beta_N,
    \end{multline*}
    where $\alpha_N$ and $\beta_N$ are given as in \eqref{eq:alpha_and_beta}.
\end{proof}

\subsection{More on Proposition \ref{prop:energy_decreasing}}
\label{appendix:D--energy_decreasing}
\begin{proposition}
    There exist a constant $\eta_N$ and an error-consistent function $\xi_N$ satisfying $\xi_N(\delta_A, \delta_B) + \eta_N < 1$, and for all $x \in \xroa$ we have
    \begin{multline*}
        \hVn(\xp) - \hVn(x) \leq \\
        -(1 - \xi_N(\delta_A, \delta_B) - \eta_N) l(x, \hat{\mu}_N(x)),
    \end{multline*}
    where the function $\xi_N$ and constant $\eta_N$ relate to the eigenvalues and matrix norms of $\hA$, $\hB$, $Q$, $R$, the input constraint set $\cU$, but not to quantities derived from $A$ or $B$. \cliu{The detailed expression of $\eta_N$ is
    \begin{equation}
    \label{eq:eta_N}
        \eta_N = \|\hA\|^2_2r_Q\gamma\rho_{\gamma}^{N-N_0},
    \end{equation}
    where $\gamma$, $\rho_{\gamma}$ and $N_0$ are defined in Lemma \ref{lm:terminal_state_exponential_bound} using a linear feedback gain $K$.} Finally, the detailed expression of $\xi_N$ is
    \begin{equation}
    \label{eq:xi_N}
        \xi_N = \omega_{N,(1)}h(\eA, \eB) + 2\omega_{N,(\frac{1}{2})}h^{\frac{1}{2}}(\eA, \eB),
    \end{equation}
    where the function $h$ is given as in \eqref{eq:prediction_error_bound_h_function} in Lemma \ref{lm:prediction_error_bound}, and the terms $\omega_{N,(1)}$ and $\omega_{N,(\frac{1}{2})}$ are given, respectively, in \eqref{eq:omega_1} and \eqref{eq:omega_0.5}.
\end{proposition}
\noindent \textbf{Details of $\omega_{N,(1)}$ and $\omega_{N,(\frac{1}{2})}$}: We define $G_N(\hA) = \sum^{N-1}_{i=1}\|\hA\|^{2(i-1)}_2$ as
\begin{equation}
\label{eq:geometric_sum}
    G_N(\hA) = 
    \begin{cases}
        N-1 & \text{if } \|\hA\|_2 = 1 \\
        \frac{1 - (\|\hA\|^2_2)^{N-1}}{1 - \|\hA\|^2_2} & \text{if } \|\hA\|_2 \neq 1,
    \end{cases}
\end{equation}
and $\omega_{N,(1)}$ and $\omega_{N,(\frac{1}{2})}$ follows as
\cliu{
\begin{subequations}
    \begin{align}
        \label{eq:omega_1}
        & \hspace*{-0.15cm}\omega_{N,(1)} \sdeq \upQ\hspace*{-0.1cm}\left[(1 \splus \|\hA\|^2_2r_Q)(\|\hA\|^2_2)^{N-1}\splus G_N(\hA)\right], \hspace*{-0.15cm}\\
        & \hspace*{-0.15cm}\omega_{N,(\frac{1}{2})} \sdeq \left[\upQ(L_{\widehat{V}}-1)G_N(\hA)\right]^{\frac{1}{2}}  \notag \\
        \label{eq:omega_0.5}
        & \hspace{3ex} + \frac{1 \splus \|\hA\|^2_2r_Q}{2}\left[\upQ(\|\hA\|^2_2)^{N-1}\gamma\rho_{\gamma}^{N-N_0}\right]^{\frac{1}{2}},
    \end{align}
\end{subequations}
}
where $L_{\hat{V}}$ is given as in Lemma \ref{lm:terminal_state_exponential_bound} and $G_N(\hA)$ is given as in \eqref{eq:geometric_sum}.

Before presenting the main proof of Proposition \ref{prop:energy_decreasing}, we need to make some preparations. By definition, we have $\hVn(x) = J_N\big(x, \hun(x)\big)$ with 
\begin{multline*}
    \hun(x) = \bigg[\big(\hun(x)\big)[0]^\top, \big(\hun(x)\big)[1]^\top, \dots, \\
    \big(\hun(x)\big)[N-1]^\top, 0\bigg]^\top.
\end{multline*}
Then, we form an auxiliary input sequence $\bvn(x)$ as
\begin{multline*}
    \bvn(x) = \bigg[\big(\hun(x)\big)[1]^\top, \big(\hun(x)\big)[2]^\top, \dots, \\
    \big(\hun(x)\big)[N-1]^\top, u_f, 0\bigg]^\top,
\end{multline*}
\cliu{
where $u_f = 0$, meaning no extra effort is spent on steering the state $\oex\big(N-1, \xp, \bvn(x)\big)$. Consequently, $\oex\left(N, \xp, \bvn(x)\right) = \hA \oex\big(N-1, \xp, \bvn(x)\big)$. To effectively compare $\hVn(\xp)$ and $\hVn(x)$, we need to quantify the state deviation due to model mismatch, which is given in the following lemma:}
\begin{lemma}
    \label{lm:trajectory_difference}
    For $k = 0,1,\dots,N-1$, the following relation holds:
    \begin{equation}
    \label{eq:state_shift_relation}
        \oex\left(k, \xp, \bvn(x)\right) = \oex\big(k+1, x, \hun(x)\big) + \hA^k \Delta x,
    \end{equation}
    where $\Delta x = (A - \hA)x + (B - \hB)\hmuN(x)$ \cliu{(as previously defined in Appendix \ref{subapp:one_step_error})}.
\end{lemma}
\begin{proof}
The proof mainly relies on the formula of open-loop state given as in \eqref{eq:state_evolution}, the shifting property of $\bvn(x)$ with respect to $\hun(x)$, and the definitions of $\xp$, $\Delta x$, as well as $\hmuN(x)$. We expand $\oex\left(k, \xp, \bvn(x)\right)$ as
    \begin{align*}
        &\oex\left(k, \xp, \bvn(x)\right) \\
        = & \hA^k\xp + \sum^{k-1}_{i=0}\hA^{k-i-1}\hB\big(\bvn(x)\big)[i] \\
        = & \hA^k\big(Ax + B\hmuN(x)\big) + \sum^{k-1}_{i=0}\hA^{k-i-1}\hB\big(\hun(x)\big)[i+1] \\
        = & \hA^kAx + \hA^kB\hmuN(x) + \sum^{k}_{i=1}\hA^{k-i}\hB\big(\hun(x)\big)[i] \\
        = & \hA^k\big[(A - \hA)x + (B - \hB)\hmuN(x)\big] + \\
        & \hspace*{15ex} + \hA^{k+1}x + \sum^{k}_{i=0}\hA^{k-i}\hB\big(\hun(x)\big)[i] \\
        = & \hA^k\Delta x + \oex\big(k+1, x, \hun(x)\big),
    \end{align*}
    which builds the relation as in \eqref{eq:state_shift_relation}.
\end{proof}

Having established the above results, we then present the main part of the proof of Proposition \ref{prop:energy_decreasing}.

\begin{proof}
    By definition of $\hVn$, we have
    \par\nobreak
    \vspace{-0.3cm}
    {\small
    \begin{align}
    \label{eq:proof_pp2_starting}
        & \hVn(\xp) - \hVn(x) \notag \\
        = & J_N\big(\xp, \hun(\xp)\big) - J_N\big(x, \hun(x)\big) \notag \\
        \leq & J_N\big(\xp, \bvn(x)\big) - J_N\big(x, \hun(x)\big) \notag \\
        \leq & \sum^{N}_{k=0}\bigg( \nQ{\oex\left(k, \xp, \bvn(x)\right)} + \nR{ \big(\bvn(x)\big)[k] } \bigg) \notag \\
        & \hspace*{3ex} - \sum^{N}_{k=0}\bigg( \nQ{\oex\big(k, x, \hun(x)\big)} + \nR{ \big(\hun(x)\big)[k] } \bigg) \notag \\
        \leq & -l\big(x, \hmuN(x)\big) \splus \nQ{ \oex\left(N\sminus1, \xp, \bvn(x)\right) } \splus \nR{u_f} \hspace{0.1cm} \splus \notag \\
        &  + \underbrace{\sum^{N-2}_{k=0}\big(\nR{ \big(\bvn(x)\big)[k]} - \nR{ \big(\hun(x)\big)[k+1] }\big)}_{= 0} + \notag \\
        & + \sum^{N-2}_{k=0}\big(\nQ{\oex\left(k, \xp, \bvn(x)\right)} \sminus \nQ{\oex\big(k\splus 1, x, \hun(x)\big)}\big)\;\splus \notag \\
        & + \nQ{\oex\left(N, \xp, \bvn(x)\right)} - \nQ{\oex\big(N, x, \hun(x)\big)},
    \end{align}
    }
    where the first inequality is due to the optimality of the input sequence $\hun(\xp)$. \cliu{Since $u_f = 0$, we further know that
    \begin{multline}
        \label{eq:proof_pp2_terminal_relation2}
            \nQ{\oex\left(N, \xp, \bvn(x)\right)} \leq \\ r_Q\ntwo{\hA}^2\nQ{ \oex\left(N\sminus1, \xp, \bvn(x)\right) }.
    \end{multline}}
    Substituting \eqref{eq:proof_pp2_terminal_relation2} into \eqref{eq:proof_pp2_starting} yields
    \par\nobreak
    \vspace{-0.3cm}
    {\small
    \begin{align}
    \label{eq:proof_pp2_cornerstone}
        & \hVn(\xp) - \hVn(x) \notag \\
        \leq & -l\big(x, \hmuN(x)\big) \sminus \nQ{\oex\big(N, x, \hun(x)\big)}  \notag \\
        & + \sum^{N-2}_{k=0}\left(\nQ{\oex\left(k, \xp, \bvn(x)\right)} \sminus \nQ{\oex\big(k\splus 1, x, \hun(x)\big)}\right)\;\notag \\
        & + \cliu{\left(1 + r_Q\ntwo{\hA}^2\right)}\nQ{ \oex\left(N\sminus1, \xp, \bvn(x)\right) }.
    \end{align}
    }%
    Using Lemma \ref{lm:quadratic_norm} and Lemma \ref{lm:trajectory_difference}, for $k = N-1$, we have
        \begin{align}
        \label{eq:proof_pp2_N-1_expansion}
        & \nQ{ \oex\left(N-1, \xp, \bvn(x)\right) } \notag \\
        \leq & \nQ{\oex\big(N, x, \hun(x)\big)} + \nQ{ \hA^{N-1}\Delta x }  \notag \\ 
        & \quad + 2\left( \nQ{\oex\big(N, x, \hun(x)\big)}\nQ{ \hA^{N-1}\Delta x } \right)^{\frac{1}{2}},
        \end{align}
    and for $k = 0,1,\dots,N-2$, we have
    \par\nobreak
    \vspace{-0.3cm}
    {\small
    \begin{align}
    \label{eq:proof_pp2_summation_expansion}
        & \sum^{N-2}_{k=0}\left(\nQ{\oex\left(k, \xp, \bvn(x)\right)} \sminus \nQ{\oex\big(k\splus 1, x, \hun(x)\big)}\right) \notag \\
        \leq & 2\left[\left(\sum^{N-1}_{k=1}\nQ{\oex\big(k, x, \hun(x)\big)} \right)\left(\sum^{N-1}_{k=1}\nQ{ \hA^{k-1}\Delta x }\right)\right]^{\frac{1}{2}}  \notag \\
        & \hspace*{5ex} + \sum^{N-1}_{k=1}\nQ{ \hA^{k-1}\Delta x }.
    \end{align}
    }%
    In addition, due to Lemma \ref{lm:prediction_error_bound}, for $k = 0,1,\dots, N-1$, we can obtain
    \begin{equation}
    \label{eq:powerA_deltaX_bound}
        \nQ{\hA^{k}\Delta x} \leq \upQ\|\hA\|^{2k}_2 h l\big(x, \hmuN(x)\big).
    \end{equation}
    Leveraging \eqref{eq:powerA_deltaX_bound} and \eqref{eq:exponential_bound_final_stage}, we can further relax \eqref{eq:proof_pp2_N-1_expansion} as
    \begin{align}
    \label{eq:proof_pp2_relax1}
        & \nQ{ \oex\left(N-1, \xp, \bvn(x)\right) } \notag \\
        \leq & \nQ{\oex\big(N, x, \hun(x)\big)} + \upQ\|\hA\|^{2(N-1)}_2 h l\big(x, \hmuN(x)\big)  \notag \\
        & + \left[\upQ\|\hA\|^{2(N-1)}_2\gamma\rho_{\gamma}^{N-N_0}\right]^{\frac{1}{2}}h^{\frac{1}{2}}l\big(x, \hmuN(x)\big).
    \end{align}
    Likewise, using \eqref{eq:powerA_deltaX_bound} and \eqref{eq:ratio_bound_hatvN}, \eqref{eq:proof_pp2_summation_expansion} can also be relaxed as
    \begin{align}
    \label{eq:proof_pp2_relax2}
        & \sum^{N-2}_{k=0}\left(\nQ{\oex\left(k, \xp, \bvn(x)\right)} \sminus \nQ{\oex\big(k\splus 1, x, \hun(x)\big)}\right) \notag \notag \\
        \leq & 2\left[\upQ(L_{\widehat{V}} - 1)G_N(\hA)\right]^{\frac{1}{2}}h^{\frac{1}{2}}l\big(x, \hmuN(x)\big)  \notag \\
        & + \upQ G_N(\hA)hl\big(x, \hmuN(x)\big),
    \end{align}
    where $G_N(\hA)$ is given as in \eqref{eq:geometric_sum}.
    
    Finally, by substituting \eqref{eq:proof_pp2_relax1} and \eqref{eq:proof_pp2_relax2} into \eqref{eq:proof_pp2_cornerstone}, rearranging the terms and applying \eqref{eq:exponential_bound_final_stage} once again, we obtain
    \begin{multline*}
        \hVn(\xp) - \hVn(x) \\
        \leq -(1 - \omega_{N,(1)}h - 2\omega_{N,(\frac{1}{2})}h^{\frac{1}{2}} - \eta_N )l\big(x, \hmuN(x)\big) \\ 
        \leq -(1 - \xi_N - \eta_N) l\big(x, \hmuN(x)\big), 
    \end{multline*}
    where $\omega_{N,(1)}$ and $\omega_{N,(\frac{1}{2})}$ are given as in \eqref{eq:omega_1} and \eqref{eq:omega_0.5}, respectively; $\xi_N$ and $\eta_N$ are given as in \eqref{eq:xi_N} and \eqref{eq:eta_N}, respectively, and the proof is completed.
\end{proof}

\cliu{
\subsection{An extension of Proposition \ref{prop:energy_decreasing}}
Choosing $u_f = 0$ is a trivial and conservative choice, and the bound obtained may not be as tight as desired (especially for unstable systems). In cases where the modeling error is small, we can choose $u_f = \hK \oex\big(N-1, \xp, \bvn(x)\big)$ with $\xp = Ax + B\hmuN(x)$ defined as in Section \ref{subsec:stability_analysis} and $\hK$ being a new stabilizable linear control gain for $(\hA, \hB)$. Note that $u_f$ is not recursively defined since $\oex\big(N-1, \xp, \bvn(x)\big)$ only requires the first $N-1$ inputs of $\bvn(x)$. We also impose an additional assumption to ensure that $u_f$ is admissible.
\begin{assumption}
    The modeling error is small enough such that there exists a linear feedback gain $\hK$ and its associated local region $\Omega_{\hK} = \{x \in \cX \mid l^\ast(x) \leq \varepsilon_{\hK} \}$ with $\varepsilon_{\hK} \geq \varepsilon_{K}$ satisfying $\oex\big(N, x, \hun(x)\big) + \hat{A}^{N-1}\Delta x \in \Omega_{\hK}$ given $\oex\big(N, x, \hun(x)\big) \in \Omega_{K}$.
\end{assumption}
The condition $\oex\big(N, x, \hun(x)\big) \in \Omega_{K}$ is guaranteed by choosing a prediction horizon $N \geq N_0$ (cf. Lemma \ref{lm:terminal_state_exponential_bound}). In fact, the deviated state $\oex\big(N-1, \xp, \bvn(x)\big) = \oex\big(N, x, \hun(x)\big) + \hat{A}^{N-1}\Delta x$ due to Lemma \ref{lm:trajectory_difference}. Since $u_f = \hK \oex\big(N-1, \xp, \bvn(x)\big)$, we further have the following relations:
\begin{subequations}
    \begin{multline}
    \label{eq:extension_proof_pp2_terminal_relation1}
        \nQ{ \oex\left(N\sminus1, \xp, \bvn(x)\right) } \splus \nR{u_f} \leq \\ C^\ast_{\hK}\nQ{ \oex\left(N\sminus1, \xp, \bvn(x)\right) }
    \end{multline}
    \begin{multline}
    \label{eq:extension_pp2_terminal_relation2}
        \nQ{\oex\left(N, \xp, \bvn(x)\right)} \leq \\ r_Q\ntwo{\hA_{\mathrm{cl}}}^2\nQ{ \oex\left(N\sminus1, \xp, \bvn(x)\right)},
    \end{multline}
\end{subequations}
where $C^\ast_{\hK}$ is defined as in Lemma \ref{lm:local_stabilization_property} with linear feedback gain $\hK$, the closed-loop gain $\hA_{\mathrm{cl}} = \hA + \hB\hK$. It should be noted that \eqref{eq:extension_pp2_terminal_relation2} is an modification of \eqref{eq:proof_pp2_terminal_relation2}. Substituting \eqref{eq:extension_proof_pp2_terminal_relation1} and \eqref{eq:extension_pp2_terminal_relation2} into \eqref{eq:proof_pp2_starting} yields
    \par\nobreak
    \vspace{-0.3cm}
    {\small
    \begin{align}
    \label{eq:proof_pp2_cornerstone_new}
        & \hVn(\xp) - \hVn(x) \notag \\
        \leq & -l\big(x, \hmuN(x)\big) \sminus \nQ{\oex\big(N, x, \hun(x)\big)}  \notag \\
        & + \sum^{N-2}_{k=0}\left(\nQ{\oex\left(k, \xp, \bvn(x)\right)} \sminus \nQ{\oex\big(k\splus 1, x, \hun(x)\big)}\right) \notag \\
        & + \left(C^\ast_{\hK} + r_Q\ntwo{\hA_{\mathrm{cl}}}^2\right)\nQ{ \oex\left(N\sminus1, \xp, \bvn(x)\right) }.
    \end{align}
}%
As a consequence, the terms $\omega_{N,(1)}$ and $\omega_{N,(\frac{1}{2})}$ become
\begin{subequations}
    \begin{align}
        \label{eq:omega_1_new}
        & \hspace*{-0.15cm}\omega_{N,(1)} \sdeq \upQ\hspace*{-0.1cm}\left[(C^\ast_{\hK} + r_Q\ntwo{\hA_{\mathrm{cl}}}^2)(\|\hA\|^2_2)^{N-1}\splus G_N(\hA)\right], \hspace*{-0.15cm}\\
        & \hspace*{-0.15cm}\omega_{N,(\frac{1}{2})} \sdeq \left[\upQ(L_{\widehat{V}}-1)G_N(\hA)\right]^{\frac{1}{2}} \notag \\
        \label{eq:omega_0.5_new}
        & \hspace{3ex} + \frac{C^\ast_{\hK} + r_Q\ntwo{\hA_{\mathrm{cl}}}^2}{2}\left[\upQ(\|\hA\|^2_2)^{N-1}\gamma\rho_{\gamma}^{N-N_0}\right]^{\frac{1}{2}},
    \end{align}
\end{subequations}
and $\eta_N$, originally given as in \eqref{eq:eta_N}, should be modified as
\begin{equation}
    \label{eq:eta_N_new}
    \eta_N = (C^\ast_{\hK} + \|\hAcl\|^2_2r_Q - 1)\gamma\rho_{\gamma}^{N-N_0}.
\end{equation}
Last but not the least, the condition \eqref{eq:horizon_requirements_original} on the desired prediction horizon is changed into
\begin{equation}
\label{eq:horizon_requirements_new}
    N > N_0 + \frac{\log[(C^\ast_{\hK} + \|\hAcl\|^2_2r_Q - 1)\gamma]}{\log(\rho_{\gamma}^{-1})}.
\end{equation}
}


\end{document}